\documentclass[10pt,a4paper]{amsart}
\usepackage[margin=1.3in]{geometry}
\usepackage{amsmath,amsthm,amssymb}
\usepackage[initials,msc-links]{amsrefs}
\usepackage{url}
\usepackage[dvipsnames]{xcolor}
\usepackage[english=american]{csquotes}
\usepackage{enumerate}
\usepackage{textcomp}
\usepackage{mathrsfs}
\usepackage{esint}
\usepackage{mathtools}
\usepackage{color, colortbl}
\definecolor{Gray}{gray}{0.9}
\usepackage{bbm}
\usepackage[multiple]{footmisc}
\usepackage{hyperref}
\usepackage{stmaryrd}
\providecommand{\keyword}[1]{{\footnotesize\textbf{\textit{Key words---}} #1}}
\hypersetup{
	bookmarks=true,         % show bookmarks bar?
	unicode=false,          % non-Latin characters in Acrobats bookmarks
	pdftoolbar=true,        % show Acrobats toolbar?
	pdfmenubar=true,        % show Acrobats menu?
	pdffitwindow=false,     % window fit to page when opened
	pdfstartview={FitH},    % fits the width of the page to the window
	pdftitle={My title},    % title
	pdfauthor={Author},     % author
	pdfsubject={Subject},   % subject of the document
	pdfcreator={Creator},   % creator of the document
	pdfproducer={Producer}, % producer of the document
	pdfkeywords={keyword1, key2, key3}, % list of keywords
	pdfnewwindow=true,      % links in new PDF window
	colorlinks=true,       % false: boxed links; true: colored links
	linkcolor=blue ,          % color of internal links (change box color with linkbordercolor)
	citecolor=blue ,        % color of links to bibliography
	filecolor=magenta,      % color of file links
	urlcolor=Aquamarine           % color of external links
}
\usepackage{caption}
\usepackage{amsfonts}
\usepackage{amssymb}
\usepackage{tikz}

\usepackage{siunitx}
\usepackage{xcolor}
\usepackage{booktabs,colortbl, array}
\usepackage{pgfplotstable}
\pgfplotsset{compat=1.8}

\definecolor{rulecolor}{RGB}{0,71,171}
\definecolor{tableheadcolor}{gray}{0.92}
\colorlet{tableheadcolor}{gray!25} % Table header colour = 25% gray
\colorlet{tablerowcolor}{gray!10} % Table row separator colour = 10% gray
\makeatletter
\DeclareFontFamily{OMX}{MnSymbolE}{}
\DeclareSymbolFont{MnLargeSymbols}{OMX}{MnSymbolE}{m}{n}
\SetSymbolFont{MnLargeSymbols}{bold}{OMX}{MnSymbolE}{b}{n}
\DeclareFontShape{OMX}{MnSymbolE}{m}{n}{
    <-6>  MnSymbolE5
    <6-7>  MnSymbolE6
    <7-8>  MnSymbolE7
    <8-9>  MnSymbolE8
    <9-10> MnSymbolE9
    <10-12> MnSymbolE10
    <12->   MnSymbolE12
}{}
\DeclareFontShape{OMX}{MnSymbolE}{b}{n}{
    <-6>  MnSymbolE-Bold5
    <6-7>  MnSymbolE-Bold6
    <7-8>  MnSymbolE-Bold7
    <8-9>  MnSymbolE-Bold8
    <9-10> MnSymbolE-Bold9
    <10-12> MnSymbolE-Bold10
    <12->   MnSymbolE-Bold12
}{}
\let\llangle\@undefined
\let\rrangle\@undefined
\DeclareMathDelimiter{\llangle}{\mathopen}%
{MnLargeSymbols}{'164}{MnLargeSymbols}{'164}
\DeclareMathDelimiter{\rrangle}{\mathclose}%
{MnLargeSymbols}{'171}{MnLargeSymbols}{'171}
\makeatother

\newtheorem{theorem}{Theorem}[section]
\newtheorem{lemma}[theorem]{Lemma}
\newtheorem{proposition}[theorem]{Proposition}
\newtheorem{corollary}[theorem]{Corollary}

\theoremstyle{definition}
\newtheorem{definition}[theorem]{Definition}
\newtheorem{remark}[theorem]{Remark}
\newtheorem{example}[theorem]{Example}

\newtheorem{assumption}[theorem]{Assumption}

\newcommand{\Trm}{\mathrm{T}}

\newcommand{\Acal}{\mathcal{A}}
\newcommand{\Bcal}{\mathcal B}

\newcommand{\Fcal}{\mathcal{F}}

\newcommand{\Hcal}{\mathcal{H}}

\newcommand{\Mcal}{\mathcal{M}}
\newcommand{\Ncal}{\mathcal{N}}

\newcommand{\Ffrak}{\mathfrak{F}}

\newcommand{\Ascr}{\mathscr{A}}

\newcommand{\Cscr}{\mathscr{C}}

\newcommand{\Rscr}{\mathscr{R}}

\newcommand{\Wscr}{\mathscr{W}}

\newcommand{\Abf}{\mathbf{A}}
\newcommand{\Bbf}{\mathbf{B}}
\newcommand{\Cbf}{\mathbf{C}}
\newcommand{\Dbf}{\mathbf{D}}
\newcommand{\Ebf}{\mathbf{E}}

\newcommand{\Gbf}{\mathbf{G}}
\newcommand{\Hbf}{\mathbf{H}}
\newcommand{\Ibf}{\mathbf{I}}

\newcommand{\Mbf}{\mathbf{M}}

\newcommand{\Sbf}{\mathbf{S}}
\newcommand{\Tbf}{\mathbf{T}}

\newcommand{\Nbb}{\mathbb{N}}

\DeclareMathOperator{\Reg}{Reg}

\DeclareMathOperator{\graph}{graph}

\DeclareMathOperator{\diverg}{div}

\DeclareMathOperator{\dist}{dist}

\DeclareMathOperator{\spn}{span}

\newcommand{\set}[2]{\left\{\, #1 \  \textup{\textbf{:}}\  #2 \,\right\}}

\newcommand{\cl}[1]{\overline{#1}}

\newcommand{\dd}{\;\mathrm{d}}

\newcommand{\N}{\mathbb{N}}
\newcommand{\R}{\mathbb{R}}

\newcommand{\loc}{\mathrm{loc}}

\newcommand{\spt}{\mathrm{spt}}

\newcommand{\Sing}{\mathrm{Sing}}

\newcommand{\todown}{\downarrow}
\newcommand{\eps}{\epsilon}

\DeclareMathOperator{\Err}{Err}

\renewcommand{\eps}{\varepsilon}
\newcommand{\vphi}{\varphi}
\makeatletter
\renewcommand*\env@matrix[1][*\c@MaxMatrixCols c]{%
    \hskip -\arraycolsep
    \let\@ifnextchar\new@ifnextchar
    \array{#1}}
\makeatother

\newcommand{\mres}{\mathbin{\vrule height 1.6ex depth 0pt width
        0.13ex\vrule height 0.13ex depth 0pt width 1.3ex}}

\def\vint_#1{\mathchoice%
    {\mathop{\kern 0.2em\vrule width 0.6em height 0.69678ex depth -0.58065ex
            \kern -0.8em \intop}\nolimits_{\kern -0.4em#1}}%
    {\mathop{\kern 0.1em\vrule width 0.5em height 0.69678ex depth -0.60387ex
            \kern -0.6em \intop}\nolimits_{#1}}%
    {\mathop{\kern 0.1em\vrule width 0.5em height 0.69678ex depth -0.60387ex
            \kern -0.6em \intop}\nolimits_{#1}}%
    {\mathop{\kern 0.1em\vrule width 0.5em height 0.69678ex depth -0.60387ex
            \kern -0.6em \intop}\nolimits_{#1}}}
\makeatletter
\newcommand*{\RangeX}{%
    {%
        \mathpalette\@RangeOf{X}%
    }%
}
\newcommand*{\@RangeOf}[2]{%
    \sbox0{$\m@th#1\mathsf{#2}$}%
    \mathsf{#2}%
    \kern-\wd0 %
    \mkern2.75mu\relax
    \nonscript\mkern.25mu\relax
    \mathsf{#2}%
}
\makeatother

\newcommand{\aveint}[2]{\mathchoice%
    {\mathop{\kern 0.2em\vrule width 0.6em height 0.69678ex depth -0.58065ex
            \kern -0.8em \intop}\nolimits_{\kern -0.45em#1}^{#2}}%
    {\mathop{\kern 0.1em\vrule width 0.5em height 0.69678ex depth -0.60387ex
            \kern -0.6em \intop}\nolimits_{#1}^{#2}}%
    {\mathop{\kern 0.1em\vrule width 0.5em height 0.69678ex depth -0.60387ex
            \kern -0.6em \intop}\nolimits_{#1}^{#2}}%
    {\mathop{\kern 0.1em\vrule width 0.5em height 0.69678ex depth -0.60387ex
            \kern -0.6em \intop}\nolimits_{#1}^{#2}}}

\allowdisplaybreaks

\title[Rectifiability of flat singularities mod$(2Q)$]{Rectifiability of flat singular points for area-minimizing mod$(2Q)$ hypercurrents}

\date{\today}

\author[A. Skorobogatova]{Anna Skorobogatova}
\address{Department of Mathematics, Fine Hall, Princeton University, Washington Road, Princeton, NJ 08540, USA}
\email{as110@princeton.edu}

\makeindex

\begin{document}
	
\maketitle

\begin{abstract}
	Consider an $m$-dimensional area minimizing mod$(2Q)$ current $T$, with $Q\in\N$, inside a sufficiently regular Riemannian manifold of dimension $m + 1$. We show that the set of singular density-$Q$ points with a flat tangent cone is $(m-2)$-rectifiable. This complements the thorough structural analysis of the singularities of area-minimizing hypersurfaces modulo $p$ that has been completed in the series of works of De Lellis-Hirsch-Marchese-Stuvard and  De Lellis-Hirsch-Marchese-Stuvard-Spolaor, and the work of Minter-Wickramasekera.
\end{abstract}

\keyword{minimal surfaces, area minimizing currents, rectifiability, regularity theory, multiple valued functions, blow-up analysis, center manifold}

\allowdisplaybreaks

\section{Introduction and main results}
Suppose that $T$ is an $m$-dimensional integer rectifiable current supported in a complete $(m+\bar{n})$-dimensional $C^1$ Riemannian submanifold $\Sigma\subset \R^{m+n}$ without boundary (for which we use the notation $T\in \Rscr_m(\Sigma)$), and let $p \geq 2$ be a given integer. 

Given an open set $\Omega \subset \R^{m+n}$ we say that $T$ is \emph{area-minimizing} mod$(p)$ in $\Sigma\cap\Omega$ if it has minimal $m$-dimensional mass in its mod$(p)$ homology class within $\Omega\cap\Sigma$, namely
\[
\Mbf(T) \leq \Mbf(T + S) \quad \text{for each $S \in\Rscr_m(\Omega\cap\Sigma)$ with $[S] = \partial^p[R]$ for some $R \in \Rscr_m(\Omega\cap\Sigma)$},
\]
where $[S] \in \Rscr_m(\Omega\cap K) / \sim_p$, for the equivalence relation $\sim_p$ given by $T\sim_p S$ if $T = S$ mod $p$. Equivalently, this can be written as
\[
\Mbf^p([T]) \leq \Mbf^p([T] + \partial^p[R]) \qquad \text{for every $R\in \Rscr_m(\Omega\cap\Sigma)$},
\]
where $\Mbf^p([T])$ is the mass mod$(p)$ for the class $[T]$, defined by
\[
\Mbf^p([T]) \coloneqq \inf\set{t \geq 0}{\substack{\text{$\forall \ \eps > 0 \ \exists$ compact $K \subset \Sigma$, $S \in \Rscr_m(\Sigma)$} \\ \text{with $\Fcal^p_K(T,S) < \eps$, $\Mbf(S) \leq t + \eps$}}}.
\]

Given $S \in \Rscr_m(\R^{m+n})$ (not necessarily a representative mod$(p)$), we let $\|S\|_p$ denote the mod$(p)$ mass measure associated with $[S]$ when identifying $[S]$ with a vector-valued Radon measure. Note that if $S$ is a representative mod$(p)$, then $\|S\|_p$ agrees with the classical $m$-dimensional mass measure $\|S\|$ associated to $S$, induced by the vector-valued Radon measure $\vec{T}\|T\|$ identified with $T$. We will henceforth make the following underlying assumption:

\begin{assumption}\label{asm:modp1}
	$T\in \Rscr_m(\R^{m+n})$ is an $m$-dimensional representative mod$(p)$ in a $C^{3,\alpha_0}$ $(m+\bar{n})$-dimensional Riemannian submanifold $\Sigma \subset \R^{m+n}$ with $\alpha_0 \in (0,1)$. $T$ is area-minimizing mod$(p)$ in $\Sigma \cap \Bbf_{7\sqrt{m}}$ for some open set $\Bbf_{7\sqrt{m}}\subset\R^{m+n}$ containing $0$ and $\partial^p [T]\mres \Bbf_{7\sqrt{m}} = 0 \ \text{mod}(p)$. Note that $\Theta(T, x) \in \left[1,\frac{p}{2}\right]$ for $\|T\|$-almost every $x$. 
	
	We may assume that $\Sigma \cap \Bbf_{7\sqrt{m}}$ is the graph of a $C^{3,\alpha_0}$ function $\Psi_p : \Trm_p\Sigma \cap \Bbf_{7\sqrt{m}} \to \Trm_p\Sigma^\perp$ for every $p \in \Sigma\cap \Bbf_{7\sqrt{m}}$. We may further assume that
	\[
	\boldsymbol{c}(\Sigma,\Bbf_{7\sqrt{m}})\coloneqq\sup_{p \in \Sigma \cap \Bbf_{7\sqrt{m}}}\|D\Psi_p\|_{C^{2,\alpha_0}} \leq \bar{\eps},
	\]
	where $\bar\eps$ will be determined later. This in particular gives us the following uniform control on the second fundamental form $A_\Sigma$ of $\Sigma$:
	\[
	\Abf_\Sigma \coloneqq \|A_\Sigma\|_{C^0(\Sigma\cap \Bbf_{7\sqrt{m}})} \leq C_0\boldsymbol{c}(\Sigma,\Bbf_{7\sqrt{m}}) \leq C_0 \bar\eps.
	\]
\end{assumption}

Given $T$ satisfying Assumption \ref{asm:modp1}, a point $p\in \spt T$ is called an (interior) regular point if there is a ball $\Bbf_r (p)$ in which $\spt T$ is an {\em embedded} submanifold of $\Sigma$ without boundary in $\Bbf_r (p)$. Its complement in $\spt T\setminus \spt (\partial T)$ is called the (interior) singular set and will henceforth be denoted by $\Sing (T)$.

Understanding the size and the structure of the singularities of $T$ in this setting was a problem first studied by Federer \cite{Federer1970} in the case $p=2$, namely, unoriented surfaces. There, it was shown that the Hausdorff dimension of the singular set is at most $m-2$, while in \cite{Simon_rectifiability, Simon_cylindrical}, Simon subsequently improved this to $(m-2)$-rectifiability and local finiteness of $(m-2)$-dimensional Hausdorff measure. Furthermore, J. Taylor~\cite{JTaylor} handled the case when $m=2$ and $p=3$ in three-dimensional ambient Euclidean space, achieving a groundbreaking structural result demonstrating that the only singularities models are superpositions of three half-planes meeting along an axis at angles of $\frac{2\pi}{3}$, and the singularities are locally a $C^{1,\alpha}$ perturbation of this. 

In light of a stratification of the singular set based of the maximal number of directions of translation-invariance of any tangent cone, the biggest obstruction to understanding the size and structure of the singularities is due to the presence of singular points with \emph{flat tangent cones} (of multiplicity at least two).

When $p$ is odd, in the work \cite{DLHMS} of De Lellis, Hirsch, Marchese and Stuvard, the authors demonstrate that the singular set of $T$ is $(m-1)$-rectifiable. The key is that in higher codimension, singularities at which $T$ admits a flat tangent cone \emph{can} arise, but the result \cite{DLHMS}*{Theorem 1.7} implies that they will necessarily have (positive integer) density \emph{strictly smaller than} $\frac{p}{2}$. Thus, such flat singularities can be dealt with inductively on the density $Q$, and we need not consider them; we only handle the highest density singular points with $Q = \frac{p}{2}$. 

When the codimension $\bar{n}=1$, the work of White~\cite{White86} tells us that any point $x \in \spt T$ with a flat tangent cone $k\llbracket \pi \rrbracket$ with integer multiplicity $k \in \left(-\frac{p}{2}, \frac{p}{2}\right)$ is necessarily regular. When $p$ is odd, this tells us that \emph{in codimension one}, there are no flat singular points. Making crucial use of this result, Taylor's structure theorem was successfully generalized in codimension one by De Lellis, Hirsch, Marchese, Spolaor and Stuvard in~\cite{DLHMSS_odd_moduli_p}, where it was demonstrated that when $p$ is odd, outside of an $(m-2)$-rectifiable set, the singular set of $T$ is locally a $C^{1,\alpha}$ $(m-1)$-dimensional submanifold, with a singularity model consiting of a superposition of $m$-dimensional half-spaces meeting in an $(m-1)$-dimensional axis.  

However, when $p$ is even, one cannot rule out the appearance of singular points of density $Q = \frac{p}{2}$ with a flat tangent cone (which we will henceforth refer to as \emph{flat singular $Q$-points}, and denote by $\Ffrak_Q (T)$). A prototypical example is as follows.

\begin{example}[cf. Figure \ref{fig:ex1} below]\label{ex:1}
	Let $f: \R^2 \to \R$ be the map $f \equiv 0$ and let $g: \R^2 \to \R$ be a (non-trivial) solution of the minimal surface equation 
	\[
	\diverg\left(\frac{\nabla g}{\sqrt{1+|\nabla g|^2}}\right) = 0,
	\]
	with $g(0) = \nabla g(0) = 0$. Let $T$ be the two-dimensional area-minimizing current mod(4) given by
	\[
	T\coloneqq \llbracket \graph (f) \cup \graph (g) \rrbracket,
	\]
	with alternating orientations in the regions between the intersections of the two graphs, to ensure that $\partial^4 [T] \mres \Bbf_1 = 0$.
	
	\begin{figure}[ht]
		\includegraphics[scale=0.4]{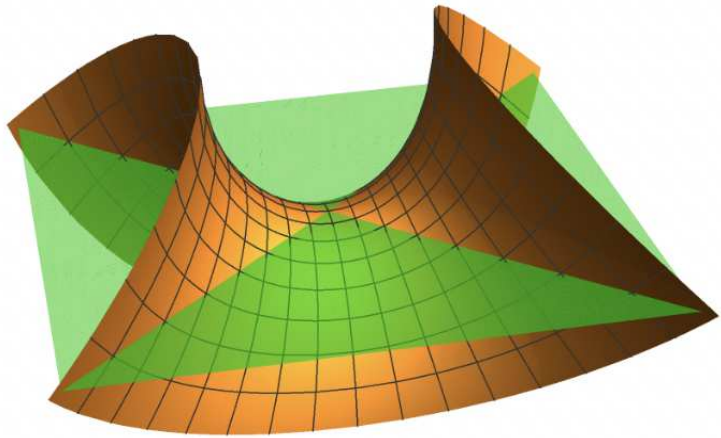}
		\caption{The graphs of $f$ and $g$ in Example \ref{ex:1}.}
		\label{fig:ex1}
	\end{figure}
	
	The origin is an isolated flat singular point of density $2$ here, meanwhile the curve segments $(\graph (f) \cap \graph (g)) \setminus \{0\}$ consist of singular points of density $2$, at each of which there is a (unique) ``open book" blowup that is the union of $4$ half-planes meeting in a line, with all of the orientations directed towards this line, to ensure that it is counted with multiplicity $4$, and thus does not create a boundary mod$(4)$.
\end{example}

In the recent works \cite{DLHMSS_structure} and \cite{MW21}, it was shown that at under the assumption that $\bar{n} = 1$ (namely, in codimension one) and with $p=2Q$ for $Q \in \Nbb$, at all flat singular $Q$-points of $T$ the flat tangent cone is unique and there is a polynomial decay rate of $T$ towards this flat tangent cone. In \cite{MW21}, the authors also establish $Q$-valued $C^{1,\alpha}$-graphicality (in a suitable sense) of $T$ locally around such points. In \cite{DLHMSS_structure}, it is additionally shown that the Hausdorff dimension of $\Ffrak_Q (T)$ is at most $m-2$, and up to a set of Hausdorff dimension at most $m-2$, the singular set of $T$ is locally a $C^{1,\alpha}$ $(m-1)$-dimensional submanifold (cf. the above discussion in the case when $p$ is odd). Note that these results heavily rely on the codimension one assumption, which allows one to classify the possible degrees of homogeneity of solutions to the linearized problem (see \cite{DLHMSS_structure}*{Proposition 2.9}). Unlike in the case when $p$ is odd, however, the authors in~\cite{DLHMSS_structure} were unable to easily establish $(m-2)$ rectifiability of the lower-dimensional part of the singular set, due to the presence of flat singularities.

In this article, we adapt the techniques developed in~\cite{DLMSV,DLSk2} in the context of higher codimension area-minimizing surfaces, in order to indeed demonstrate the $(m-2)$-rectifiability and local finiteness of $(m-2)$-dimensional Hausdorff measure for the flat singular set of $T$ when $\bar{n}=1$ and $p$ is even:

\begin{theorem}\label{thm:modpmain}
	Let $T$ satisfy Assumption~\ref{asm:modp1} (c.f.~\cite{DLHMS}*{Assumptions 17.5}) with $\bar n = 1$ and $p = 2Q$ for some $Q \in \N$. If the parameters in~\cite{DLHMS}*{Assumption~17.11} are chosen appropriately, then $\Ffrak_Q (T)$ is $(m-2)$-rectifiable.
\end{theorem}

This, in particular, implies that when $\bar{n}=1$ and $p=2Q$, up to an $(m-2)$-rectifiable set, the singular set of $T$ is locally a $C^{1,\alpha}$ $(m-1)$-dimensional submanifold.

\subsection{Connection with stable minimal hypersurfaces}
Any $T$ satisfying Assumption \ref{asm:modp1} induces a stable integral varifold. Moreover, the existing known regularity theory for stable integral varifolds of codimension one ($\bar{n}=1$) appears to be consistent with the known regularity theory for area-minimizing mod$(p)$ hypersurfaces; see the works \cite{W14} and \cite{MW21}. In the more general framework of codimension one $m$-dimensional stable integral varifolds, the work \cite{KW} of Krummel and Wickramasekera demonstrates that locally in the regions where there are no points of density $3$ or higher, the interior flat singular set is $(m-2)$-rectifiable. However, to the knowledge of the author, such a result is still open for higher multiplicities, in such a framework.

\section*{Acknowledgments} The author is indebted to Professor Camillo De Lellis for introducing her to this problem, taking the time to explain important background results to her, and reading a preliminary draft of the article. The author would also like to further thank both Camillo De Lellis and Vikram Giri for partaking in many fruitful discussions with her. She would also like to thank Paul Minter for pointing out some known results in the literature that she was not aware of.

The author acknowledges the support of the National Science Foundation through the grant FRG-1854147.

\section{Notation and preliminaries}\label{s:prelim}
We begin this section by providing a list of notation, consistent with \cite{Federer,DLHMS,DLHMSlin}, which will be frequently used throughout this article.
\begin{align*}
	%	& \Rscr_m(\Omega) && \text{the space of $m$-dimensional integer rectifiable currents in $\Omega$}; \\
	%	& \Rscr_m(K) && \text{the elements of $\Rscr_m(\R^{m+n})$ with compact support in $K$}; \\
	&\Fcal^p_K(S,T) &&\text{the flat distance modulo $p$ between the $m$-dimensional integral flat chains $S, T$} \\
	& \ && \text{with compact support in $K$ (see, for example~\cite{Federer}*{Section 4.2.26} or~\cite{DLHMS});}\\
	&\partial^p [T] && \text{the boundary modulo $p$ of $[T]$, defined by $\partial^p [T] \coloneqq [\partial T]$}; \\
	%	&\Mbf(T) &&\text{the $m$-dimensional mass of $T \in \Rscr_m$;} \\
	%	&\|T\| &&\text{the $m$-dimensional mass measure associated to $T$, when identifying $T$} \\
	%	& \ &&\text{with the vector-valued Radon measure $\vec{T}\|T\|$}; \\
	&\spt^p(T) &&\text{the support mod $p$ of $T \in \Rscr_m$, defined by $\spt^p(T) \coloneqq \bigcap_{S = T \ \textrm{mod}(p)} \spt S$}; \\
	& \Acal_Q(\R^k) && \text{the space of $Q$-tuples of vectors in $\R^k$ (see \cite{DLS_MAMS} for more details);} \\
	& \Ascr_Q(\R^k) && \text{the quotient space $\Acal_Q(\R^k)\times \{-1,1\} / \sim$, where $\sim$ is the equivalence relation} \\
	& \ && \text{given by $(S,-1) \sim (T,1) \iff \exists q \in \R^k$ with $S=Q\llbracket q\rrbracket = T$, and} \\ & \ &&(S,\pm 1) \sim (T, \pm 1) \iff S = T; \\
	& \Bbf_r(p) && \text{the $(m+n)$-dimensional Euclidean ball of radius $r$ centered at $p$};\\
	& \Bcal_r(z) && \text{the geodesic ball of radius $r$ centered at $z$ on a given center manifold (see~\cite{DLS16centermfld}} \\
	& \ && \text{for more details);}\\
	&\Hcal^s &&\text{the} \ s \text{-dimensional Hausdorff measure}, \ s \geq 0; \\
	& d_H && \text{the Hausdorff distance, defined on the space of compact subsets of $\R^{m+n}$}; \\
	&W^{k,s}(\Omega;\Acal_Q) &&\text{the space of} \ Q\text{-valued} \ s\text{-integrable Sobolev maps with $s$-integrable} \\ 
	& \ && \text{distributional derivatives up to order $k \in \Nbb$} \ \text{on} \ \Omega; \\
	&B_r(z,\pi) &&\text{the $m$-dimensional Euclidean ball of radius $r$ and center $z$ in the} \\
	& \ &&\text{$m$-dimensional plane $\pi$. If it is clear from context, we will just write $B_r(z)$;} \\
	& E^\perp &&\text{The orthogonal complement to the set $E$ with respect to the standard } \\
	& \ &&\text{Euclidean inner product;} \\
	&\Cbf_r(z,\pi) &&\text{the infinite $(m+n)$-dimensional Euclidean cylinder $B_r(z,\pi) + \pi^\perp$ with center}\\
	& \ &&\text{$z$, radius $r$ in direction $\pi^\perp$;} \\
	& \iota_{z,r} &&\text{the scaling map $w \mapsto \frac{w-z}{r}$ around the center $z$;} \\
	& \tau_z &&\text{the translation map $w \mapsto w+z$;} \\
	& f_\sharp &&\text{the push-forward under the map $f$;} \\
	& E_{z,r} &&\text{the blow-up $(\iota_{z,r})_\sharp E$ of the set $E$;} \\
	&T_p \Ncal &&\text{the tangent plane to the manifold $\Ncal$ at the point $p \in \Ncal$;} \\
	&\Tbf_{F} &&\text{the current $\sum_{i\in\N}\sum_{j=1}^Q (f^j_i)_\sharp \llbracket M_i \rrbracket$ induced by the push-forward of a}\\
	& \ &&\text{$Q$-valued map $F:M \to \Acal_Q(\R^{m+n})$ on a Borel set $M \subset \R^m$ with} \\
	& \ &&\text{decomposition $F|_{M_i} = \sum_{j=1}^Q \llbracket f^j_i \rrbracket$, $M = \sqcup_i M_i$ as in~\cite{DLS_multiple_valued}*{Lemma 1.1}} \\
	& \ &&\text{(see~\cite{DLS_multiple_valued}*{Section 1.1} for a more detailed definition);} \\
	%	& E \Subset F &&\text{The set $E$ is compactly contained in the set $F$, namely, $\cl{E} \subset F$;} \\
	%	& A \simeq B &&\text{The quantities $A$ and $B$ are comparable, namely, $c_1A \leq B \leq c_2B$} \\
	%	& \ &&\text{for some $c_1, c_2 > 0$}; \\
	&\Theta(T,p) &&\text{the $m$-dimensional Hausdorff density of $T$ at a given point $p$;} \\
	&\mathbf{p}_\pi &&\text{the orthogonal projection to the $m$-plane $\pi \subset \R^{m+n}$;} \\
	%	&\Fbf(V_1,V_2) &&\text{the varifold distance (induced by the weak-$*$ topology) between $V_1$} \\
	%	&\ &&\text{and $V_2$}. \\
	%	&\fint_E &&\text{the averaged integral $\frac{1}{|E|}\int_E$}; \\
	&T_{x,r} &&\text{the pushforward $(\iota_{x,r})_\sharp T$ of $T$ under the rescaling map $\iota_{x,r}$}.
\end{align*}

We are now in a position to introduce some key notions that will be pivotal in this article. In order to study the behaviour of $T$ around flat singularities, one needs Almgren's celebrated center manifold construction (see~\cite{Almgren_regularity,DLS16centermfld}) and the linear theory of \emph{special $Q$-valued maps} in this setting. We refer the reader to~\cite{DLHMS, DLHMSlin} for the relevant background theory and notation in the mod$(p)$ framework. We recall here the definition of the \emph{non-oriented excess} of $T$ with respect to flat planes here, for the convenience of the reader.

\begin{definition}
	For $T$ as in Assumption~\ref{asm:modp1}, we define the non-oriented excess $\Ebf^{no}(T,\Cbf_r(x),\pi)$ of $T$ in $\Cbf_r(x) = B_r(x,\pi) \times \pi^\perp$ by
	\[
	\Ebf^{no}(T,\Cbf_r(x)) \coloneqq \frac{1}{2\omega_m r^m} \int_{\Cbf_r(x)} |\vec{T} - \vec{\pi}|^2_{no} \dd \|T\|,
	\]
	where
	\[
	|\vec{T} - \vec{\pi}|_{no} \coloneqq \min\{|\vec{T} - \vec{\pi}|, |\vec{T} + \vec{\pi}|\}.
	\]
	The non-oriented excess $\Ebf^{no}(T,\Bbf_r(x),\pi)$ of $T$ in $\Bbf_r(x)$ with respect to the $m$-dimensional plane $\pi$ is defined analogously. The non-oriented excess $\Ebf^{no}(T,\Bbf_r(x))$ of $T$ in $\Bbf_r(x)$ is then defined as
	\[
	\Ebf^{no}(T,\Bbf_r(x)) \coloneqq \min_{\text{$m$-planes $\pi$}} \Ebf^{no}(T,\Bbf_r(x),\pi).
	\]
	%    We say that an $m$-dimensional plane $\pi_0$ \emph{optimizes the non-oriented excess} of $T^0$ in $\Bbf_r(x)$ if $\Ebf^{no}(T^0,\Bbf_r(x)) = \Ebf^{no}(T^0,\Bbf_r(x),\pi_0)$.
	We define the mod$(p)$ excess $\Ebf(T,\Cbf_r(x))$ of $T$ in $\Cbf_r(x) = B_r(x,\pi) \times \pi^\perp$ in analogous manner to the classical oriented excess for integral currents:
	\[
	\Ebf(T,\Cbf_r(x)) \coloneqq \frac{1}{2\omega_m r^m} \int_{\Cbf_r(x)} |\vec{T} - \vec{\pi}|^2\dd\|T\| = \frac{\| T\|(\Cbf_r(x)) - \|\mathbf{p}_\sharp T\|_p (\Cbf_r(x))}{\omega_m r^m}.
	\]
	Note that although $\|T\| = \|T\|_p$ since $T$ is a representative mod$(p)$, this is no longer necessarily the case for $\|\mathbf{p}_\sharp T\|$ and $\|\mathbf{p}_\sharp T\|_p$, since $\mathbf{p}_\sharp T$ need not be a representative mod$(p)$. 
	%\note{A: how can this happen in practice? Is it like if you have p sheets coming together over some plane, so that it's $0 \mod (p)$ but when you project to the plane and take the mass you get multiplicity $p$?} 
	The mod$(p)$ excess $\Ebf(T,\Bbf_r(x),\pi)$ of $T$ in $\Bbf_r(x)$ with respect to the $m$-dimensional plane $\pi$ and the mod$(p)$ excess $\Ebf(T,\Bbf_r(x))$ are defined analogously.
\end{definition}

\subsection{Reduction to a single center manifold}
Following \cite{DLHMS}*{Section 25} we introduce appropriate disjoint intervals $]s_j, t_j]\subset ]0,1]$, called {\em intervals of flattening}, the union of which identifies those radii $r$ such that the non-oriented excess $\mathbf{E}^{no} (T, \Bbf_{6\sqrt{m} r})$ falls below a positive fixed threshold $\varepsilon_3^2$. Arguing as in \cite{DLHMS}*{Section 25} for each rescaled current $T_{0, t_j}$ and rescaled ambient manifold $\Sigma_{0, t_j}$ we produce a center manifold $\mathcal{M}_j$ and an appropriate multivalued map $N_j : \mathcal{M}_j \to \mathscr{A}_Q (\mathbb R^{m+n})$. The latter takes values in the normal bundle of $\mathcal{M}$ and gives an efficient approximation of the current $T_{0, t_j}$ in $\Bbf_3\setminus \Bbf_{s_j/t_j}$.
We recall the following (non-oriented) tilt excess decay from~\cite{DLHMSS_structure}:
\begin{proposition}[\cite{DLHMSS_structure}*{Proposition 2.3}]\label{prop:excessdecaymodp}
	Let $\delta_2 > 0$ be fixed as in \cite{DLHMS}*{Assumption 17.10}. There exists $\eps_1 = \eps_1(\delta_2,m,Q)$ and a positive constant $C = C(\delta_2,m,Q)$ such that the following holds. Suppose that $T$ satisfies Assumption~\ref{asm:modp1} with $\bar n = 1$ and $p=2Q$. Assume that $q\in \Ffrak_Q(T)$ and that 
	\begin{equation}\label{e:excessthreshold}
		\mathbf{E}^{no} (T, \mathbf{B}_\rho(q),\pi_q) + (\rho\Abf)^2 < \varepsilon_1, \qquad \Cbf_{\rho/4}(q, \pi_q)\cap\spt_p T \subset\Bbf_{\rho/2}(q).
	\end{equation} 
	Then for every $r\leq \frac{\rho}{32}$, we have
	\begin{align}\label{e:excessdecay-quantitative-higherI}
		\Ebf^{no}(T,\Bbf_r(q)) \leq C\left(\frac{r}{\rho}\right)^{2-2\delta_2}\left(\Ebf^{no}(T,\Bbf_\rho(q),\pi_q) +(\rho\Abf)^2\right).
	\end{align}
\end{proposition}
Moreover, by an obvious adaptation of the proof of \cite{DLHMSS_structure}*{Proposition 2.4}, we have the following.

\begin{proposition}[cf. \cite{DLHMSS_structure}*{Proposition 2.4}]\label{p:onecm}
	For every $\bar{c}_s>0$, there is a constant $\tau_0=\tau_0(m,Q,\bar{c}_s)>0$ with the following property. Let $T$ satisfy Assumption~\ref{asm:modp1} with $\bar n = 1$, $p=2Q$ and let $\delta_2 > 0$ be fixed as in \cite{DLHMS}*{Assumption 17.10}. Then there exists a choice of parameters in \cite{DLHMS}*{Assumption 17.11} with this choice of $\delta_2$, such that if the assumptions \cite{DLHMS}*{Assumption 17.5} hold and \eqref{e:excessthreshold} holds with $\rho=6\sqrt{m}$, then
	\begin{itemize}
		\item[(a)] The decay \eqref{e:excessdecay-quantitative-higherI} holds for every $r\leq \tau_0$;
		\item[(b)] $\Ffrak_Q(T)\cap\Bbf_{\tau_0}$ is a subset of $\mathbf\Phi(\mathbf{\Gamma})\subset \mathcal{M}_0$;
		\item[(c)] for each $q=(x_q,y_q)\in \Ffrak_Q(T)\cap\Bbf_{\tau_0}$, where $x_q\in\pi_0$, $y_q\in\pi_0^\perp$, we have
		\[
		L\in\Wscr \qquad\implies \qquad \ell(L)< \bar{c}_s \dist(x_q,L).
		\]
	\end{itemize}
\end{proposition}
Indeed, notice that in the statement \cite{DLHMSS_structure}*{Proposition 2.4}, one may replace $\frac{1}{64\sqrt{m}}$ with a choice of $\bar{c}_s>0$ arbitrarily small, at the price of allowing the scale $\tau_0$ to additionally depend on $\bar{c}_s$. We do not include the details here, and refer the reader to the proof therein.

Now fix $\bar{c}_s>0$ (to be determined in Theorem \ref{t:modpreduction} below) and let us decompose $\Ffrak_Q(T)$ into a countable union of (nested) sets as follows:
\[
\Sbf_j \coloneqq \left\{q\in \Ffrak_Q(T): \text{the assumptions of Proposition \ref{p:onecm} hold with $\bar{c}_s$ for $T_{q, \frac{1}{6\sqrt{m}j}}$} \right\}.
\]
Now given any point $q\in\Sbf_j$, we may apply Proposition \ref{p:onecm} with this choice of $\bar{c}_s$ to $T_{q,\frac{1}{6\sqrt{m}j}}$ to reduce Theorem \ref{thm:modpmain} to the following theorem, the first three conclusions of which are an immediate consequence of the above discussion (without any constraint on $\eps_4$ or $\eta$). Note that here we use a slightly different definition of $\boldsymbol{m}_0$ to that in \cite{DLHMS}, but the arguments therein remain unchanged under such a replacement (up to possibly decreasing the parameter $\eps_2$ therein).
\begin{theorem}\label{t:modpreduction}
	There exists $\eps_4=\eps_4(Q,m)> 0$, $\eta=\eta(Q,m)>0$ such that for some $\bar{c}_s=\bar{c}_s(m,\eta)>0$, the following holds.  Let $T$ satisfy Assumption~\ref{asm:modp1} with $\bar n = 1$, $p=2Q$ and let $\delta_2 > 0$ be fixed as in \cite{DLHMS}*{Assumption 17.10}. Then for any $j\in \N$ and any $q\in \mathbf{S}_j$, letting $r_0 \coloneqq \frac{\tau_0}{6\sqrt{m}j}$ for $\tau$ as in Proposition \ref{p:onecm} and $\boldsymbol{m}_0\coloneqq \mathbf{E}^{no} (T_{q,r_0}, \mathbf{B}_{6\sqrt{m}}) + (6\sqrt{m}\Abf)^2 < \eps_4^2$, there exists a choice of parameters in \cite{DLHMS}*{Assumption 17.11} with this choice of $\delta_2$, such that the following properties hold:
	\begin{itemize}
		\item[(i)] $T_{p, r_0}$ satisfies the assumptions of the relevant statements in \cite{DLHMS} (in place of $T$), where the center manifold $\mathcal{M}_0$ is constructed using $\delta_2$ and $\boldsymbol{m}_0$ as defined above;
		\item[(ii)] the decay \eqref{e:excessdecay-quantitative-higherI} holds  for $T_{q,r_0}$ for all scales $r\leq 1$;
		\item[(iii)] the rescaling $\iota_{q,r_0}(\mathbf{S}_j) \cap \overline{\Bbf}_{6\sqrt{m}}$, and thus also its closure $\Sbf$ (for which we omit dependency on $j$), is contained in $\mathcal{M}_0$; 
		\item[(iv)] for every $x\in \Mcal_0\cap\Bbf_{6\sqrt{m}}$ and for every $r\in ]\eta d(x,\Sbf),1]$ (where $d (x, \mathbf{S}) = \min \{ d(x,y): y\in \mathbf{S}\}$), every cube $L$ which intersects $B_r (q, \pi_0)$ satisfies $\ell (L) \leq c_s r$, where $c_s=\textstyle{\frac{1}{64\sqrt{m}}}$ is as in \cite{DLHMS}*{(25.5)}. 
		\item[(v)] $\mathbf{S}$ has finite upper $(m-2)$-dimensional Minkowski content and it is $(m-2)$-rectifiable.
	\end{itemize} 
\end{theorem}

\begin{proof}[Proof of Theorem \ref{t:modpreduction}(iv)]
	For any point $x\in \Sbf$, the conclusion follows immediately from conclusion (c) of Proposition \ref{p:onecm} with $\bar{c}_s=\eta$. It remains to verify the conclusion at points outside of $\Sbf$. Taking $r \in ]\eta d (x, \mathbf{S}), 1]$ and $c_s$ as in the statement of the theorem, observe that any cube $L\in\Cscr$ with $L\cap B_r(q,\pi_0)\neq\emptyset$ and $\ell(L) > c_s r > \bar c_s \eta d (x, \mathbf{S})$ would in turn satisfy $L\cap B_{d (x, \mathbf{S}) + r}(\tilde{q},\pi_0)\neq \emptyset$ for some $\Sbf\ni \tilde{x} = \mathbf{p}_{\pi_0}(\tilde{q})$, contradicting conclusion (c) of Proposition \ref{p:onecm} for the choice $\bar{c}_s = \frac{c_s}{1+\frac{1}{\eta}}$.
\end{proof}

The remainder of this article will thus be dedicated to proving the conclusion (v) of Theorem \ref{t:modpreduction}, from which the conclusion of Theorem \ref{thm:modpmain} follows immediately by an elementary covering argument. The conclusion of \ref{t:modpreduction}(v) will be given in Section \ref{s:concl}, at the very end of the article.

%\begin{remark}\label{rmk:non-contact-cubes}
%	Note that although for any point in $\Sbf$ we may ensure that the result of Corollary \ref{c:small-cubes} holds at all scales $r\in]0,1]$, we cannot hope to achieve this at all scales $r\in]0,1]$ at points that are not in $\mathbf{\Phi}(\mathbf{\Gamma})$. Unfortunately, we need to consider points outside of the contact set when taking spatial variations. Nevertheless, for an arbitrary small constant $\eta > 0$, we can leverage Corollary \ref{c:small-cubes} to establish the same conclusion at any given point $x\in \Mcal\cap\Bbf_1$ and every scale $r \in ]\eta d (x, \mathbf{S}),1]$ (where, as usual $d (x, \mathbf{S}) = \min \{ d(x,y): y\in \mathbf{S}\}$). This is due to the observation that for any $\bar c_s >0$, any cube $L\in\Cscr$ with $L\cap B_r(q,\pi_0)\neq\emptyset$ and $\ell(L) > \bar c_s r > \bar c_s \eta d (x, \mathbf{S})$ would in turn satisfy $L\cap B_{d (x, \mathbf{S}) + r}(\tilde{q},\pi_0)\neq \emptyset$ for some $\Sbf\ni \tilde{x} = \mathbf{p}_{\pi_0}(\tilde{q})$, contradicting the conclusion of Corollary \ref{c:small-cubes}.
%\end{remark}

Translating by $q$ and rescaling so that $r_0$ is taken to be unit scale, and henceforth denoting $\Mcal_0$ by simply $\Mcal$, we may thus make the following assumptions from now on. 

\begin{assumption}\label{asm:modp2}
	For some fixed (yet arbitrary) positive constants $\varepsilon_4$ and $\eta$, the following holds.
	\begin{itemize}
		\item[(i)] $T$ satisfies Assumption~\ref{asm:modp1} with $\bar n = 1$, $p=2Q$ and $0 \in  \Ffrak_Q(T)$.
		\item[(ii)] There is one interval of flattening $]0,1]$ around $0$ with corresponding $\boldsymbol{m}_{0,0} \equiv \boldsymbol{m}_0 \coloneqq \mathbf{E}^{no} (T, \mathbf{B}_{6\sqrt{m}}) + (6\sqrt{m}\Abf)^2 \leq \varepsilon_4^2$.
		\item[(iii)] If $\Mcal$ is the corresponding center manifold with normal approximation $N$, then $\mathbf{S} :=  \Ffrak_Q(T)\cap \Bbf_{1}$ is contained in the contact set $\mathbf{\Phi}(\mathbf{\Gamma})$ of $\Mcal$ and the excess decay of Proposition \ref{prop:excessdecaymodp} holds at each $q \in \Sbf$, for all scales $r\in]0,1]$ and $Q\llbracket T_q\Mcal \rrbracket \mres\Bbf_{1}$ is the unique flat tangent cone of $T$ at any such $q$.
		\item[(iv)] For every $x\in \mathbf{B}_{1} \cap \mathcal{M}$, the conclusion (iv) of Theorem \ref{t:modpreduction} is valid for all radii $r \in ]\eta\, d (x, \mathbf{S}),1]$ (and hence for all radii $r \in ]0,1]$ when $x \in \Sbf$).
	\end{itemize}
\end{assumption}

\section{Frequency function and radial variations}

Let us begin by introducing the (regularized) frequency function for the $\Mcal$-normal approximation $N$ of $T$ as in Assumption \ref{asm:modp2}. Let $\phi:[0,\infty[ \to \R$ be defined by
\[
\phi (t) =
\left\{
\begin{array}{ll}
	1 \qquad &\mbox{for $0\leq t \leq \frac{1}{2}$}\\
	2-2t \quad &\mbox{for $\frac{1}{2}\leq t \leq 1$}\\
	0 &\mbox{otherwise}\, .
\end{array}
\right.
\]
Let $d:\Mcal \times \Mcal \to [0,\infty[$ denote the geodesic distance on $\Mcal$. We recall the following properties of $d$, which are consequences of the $C^{3,\kappa}$-estimates on the center manifold $\Mcal$ (we refer the reader to \cite{DLS16centermfld} and \cite{DLDPHM} for the details):
\begin{enumerate}[(i)]
	\item $d(x,y) = |x-y| + O\big(\boldsymbol{m}_0^{\frac{1}{2}} |x-y|^2\big)$, 
	\item $|\nabla_y d| = 1 + O\big(\boldsymbol{m}_0^{\frac{1}{2}}d \big)$,
	\item $\nabla^2_y (d^2) = g + O(\boldsymbol{m}_0 d)$, where $g$ is the metric induced on $\mathcal{M}$ by the Euclidean ambient metric.
\end{enumerate}
We now define the following quantities:
\begin{align*}
	\Dbf(x,r) &\coloneqq \int_\Mcal |DN|^2 \phi\left(\frac{d(y, x)}{r}\right) \dd y\, , \\
	\Hbf(x,r) &\coloneqq - \int_\Mcal \frac{|\nabla_y d(y, x)|^2}{d(y, x)} |N|^2\phi'\left(\frac{d(y, x)}{r}\right) \dd y\, \\
	\Ibf(x,r) &\coloneqq \frac{r \Dbf(x,r)}{\Hbf(x,r)}\, .
\end{align*}
We often omit the dependency on $N$ of these quantities, since we are considering one single fixed center manifold $\Mcal$ and associated normal approximation $N$ throughout. However, when it becomes necessary to highlight such dependence (e.g. in compactness arguments), we will write $\Ibf_N$, $\Dbf_N$ and $\Hbf_N$. We refer the reader to~\cite{DLS16blowup} or~\cite{DLDPHM} for more details on the above quantities and their basic properties. Moreover, since in practically all the computations the derivative of $d$ is taken in the variable which is the same as the intergration variable, in all such cases we will write instead $\nabla d$.

We will also need to use the above quantities for Dir-minimizing maps $u: \R^m \supset \Omega \to \Ascr_Q$. For such a map $u$ we define
\begin{align*}
	D_u(x,r) &\coloneqq \int_\Omega |Du|^2 \phi\left(\frac{|y-x|}{r}\right) \dd y\, , \\
	H_u(x,r) &\coloneqq - \int_\Omega \frac{1}{|y-x|} |u|^2\phi'\left(\frac{|y-x|}{r}\right) \dd y\, \\
	I_u(x,r) &\coloneqq \frac{r D_u(x,r)}{H_u(x,r)}\, .
\end{align*}

In addition, we define the scale invariant quantities
\[
\bar{\Hbf}(x,r) \coloneqq r^{-(m-1)}\Hbf(x,r), \qquad \bar{\Dbf}(x,r) \coloneqq r^{-(m-2)} \Dbf(x,r)
\]
and
\begin{align*}
	\Ebf(x,r) &\coloneqq -\frac{1}{r} \int_{\Mcal} \phi'\left(\frac{d(x, y)}{r}\right)\sum_i N_i(y)\cdot DN_i(y)\nabla d(x, y) \dd y\, , \\
	\Gbf(x,r) &\coloneqq -\frac{1}{r^2} \int_{\Mcal} \phi'\left(\frac{d(x, y)}{r}\right) \frac{d(x, y)}{|\nabla d(x, y)|^2} \sum_i |DN_i(y) \cdot \nabla d(x, y)|^2 \dd y\, , \\
	\mathbf{\Sigma}(x,r) &\coloneqq \int_{\Mcal} \phi\left(\frac{d(x, y)}{r}\right)|N(y)|^2\dd y\, .
\end{align*}
We will require the following important lemma, which verifies that the variational identities required for the almost monotonicity of the frequency function $r\mapsto\mathbf{I} (x,r)$ hold indeed for every $x\in \Mcal\cap\Bbf_1$ and for every $r \in ]0,1]$.

\begin{lemma}\label{lem:firstvar}
	There exists $\gamma_4 (m,Q) > 0$ sufficiently small and a constant $C (m,Q) > 0$ such that the following holds. Suppose that $T$, $\Mcal$, $N$ are as in Assumption \ref{asm:modp2}. Then for any $x \in \Mcal\cap\Bbf_1$ and any $r \in ]\eta d(x,\Sbf), 1]$, we have the following identities
	\begin{align*}
		\partial_r \Dbf(x,r) &= - \int_{\Mcal} \vphi'\left(\frac{d(x, y)}{r}\right) \frac{d(x, y)}{r^2} |DN(y)|^2 \ dy \\
		\partial_r \bar\Hbf(x,r)&= O(\boldsymbol{m}_0)\bar{\Hbf}(x,r) + 2r^{-(m-1)}\Ebf(x,r).
	\end{align*}
	\begin{align*}
		&|\Dbf(x,r) - \Ebf(x,r)| \leq \sum_{j=1}^5 |\Err_j^o| \leq C\boldsymbol{m}_0^{\gamma_4}\Dbf(x,r)^{1+\gamma_4} + C\boldsymbol{m}_0\mathbf{\Sigma}(x,r), \\
		&\left|\partial_r\bar\Dbf(x,r) - 2r^{-(m-2)}\Gbf(x,r)\right| \leq 2r^{-(m-2)} \sum_{j=1}^5 |\Err_j^i| + C\boldsymbol{m}_0\bar\Dbf(x,r) \\
		&\qquad \leq Cr^{-1}\boldsymbol{m}_0^{\gamma_4}\bar\Dbf(x,r) + Cr^{-(m-2)+\gamma_4}\boldsymbol{m}_0^{\gamma_4}\Dbf(x,r)^{\gamma_4}\partial_r \Dbf(x,r) +C\boldsymbol{m}_0\bar{\Dbf}(x,r),
	\end{align*}
	where $\Err_j^o$ and $\Err_j^i$ are as in~\cite{DLDPHM}*{Proposition~9.8,~Proposition~9.9}.
\end{lemma}

A simple consequence is the following quantitative almost-monotonicity for the frequency (cf. \cite{DLHMSS_structure}*{Proposition 2.7}).

\begin{corollary}\label{cor:freqmono}
	Let $T$, $\Mcal$, $N$ and $\gamma_4$ be as in Lemma \ref{lem:firstvar}. Then there exists $C=C(m,Q)> 0$ such that for any $x\in \Mcal\cap\Bbf_1$ and any $r\in ]\eta d(x,\Sbf),1]$ we have 
	\begin{align*}
		\partial_r\left[1+\log\Ibf(x,r)\right] &\geq -C\boldsymbol{m}_0^{\gamma_4}.
	\end{align*}
\end{corollary}

We omit the proof of Lemma~\ref{lem:firstvar} here, since it involves a mere repetition of the arguments in the proofs of~\cite{DLHMS}*{Proposition~26.4} (see also~\cite{DLS16blowup}*{Proposition 3.5} and \cite{DLDPHM}*{Proposition~9.5,~Proposition~9.10} for the analogous arguments in the integral currents framework), combined with the following observations:
\begin{itemize}
	\item[(1)] one may ensure that the constants therein are optimized to depend on appropriate powers of $\boldsymbol{m}_0$, resulting in the more explicit estimates given above;
	\item[(2)] the validity of the estimates in the lemma at given scale $r$ and a given point $x\in \Mcal\cap\Bbf_1$ merely uses the validity of conclusion (iv) of Theorem \ref{t:modpreduction} at this scale.
\end{itemize}
Meanwhile, for the proof of Corollary \ref{cor:freqmono}, we refer the reader to \cite{DLSk1}*{Corollary 3.5} for the analogous argument in the setting of integral currents, combined with the observation that the proof remains unchanged in the current framework, in light of Lemma \ref{lem:firstvar}.

The bound of Corollary \ref{cor:freqmono} in turn gives a uniform bound for the frequency $\mathbf{I} (x, 4)$ for each $x$ in $\mathbf{B}_1 \cap \mathcal{M}$. In particular, given the validity of the monotonicity of $\mathbf{I}$, we can infer the following upper bound 
\begin{equation}\label{e:bound-Lambda}
	\mathbf{I} (x, r) \leq \Lambda \qquad \forall x\in \mathcal{M}\cap \mathbf{B}_1\, , \forall r \in ]\eta\, d (x, \mathbf{S}), 4],
\end{equation}  
for a suitable constant $\Lambda > 0$ which depends on $T$. Before we proceed, let us first simplify the variational errors in Lemma \ref{lem:firstvar} and record some other useful estimates that will be useful in later sections.

\begin{lemma}\label{lem:simplify}
	For any fixed $\eta>0$ and $\Lambda>0$ as in \eqref{e:bound-Lambda}, if $\varepsilon_4$ is chosen sufficiently small, then there exists a constant $C=C(m,Q,\Lambda,\eta)>0$ (independent of $\eps_4$), the following holds for any $T$ as in Assumption \ref{asm:modp2}, every $x\in \Mcal\cap\Bbf_1$ and any $\rho, r \in ]\eta d (x, \mathbf{S}), 4]$. 
	\begin{align}
		C^{-1} \leq & \Ibf (x,r) \leq \Lambda \label{eq:simplify1} \\  
		\Lambda^{-1} r \Dbf (x,r) \leq & \Hbf (x,r) \leq C r \Dbf (x,r) \label{eq:simplify2} \\
		\mathbf{\Sigma} (x,r) &\leq C r^2 \Dbf (x,r) \label{eq:simplify3} \\
		\Ebf (x,r) &\leq C \Dbf (x,r) \label{eq:simplify4} \\
		\bar\Hbf(x,\rho) &= \bar\Hbf(x,r)\exp\left(-C\int_\rho^r \Ibf(x,s)\frac{\dd s}{s} - O (\boldsymbol{m}_0) (r-\rho)\right) \label{eq:simplify5} \\
		\Hbf (x, r) & \leq C \Hbf (x, \textstyle{\frac{r}{4}}) \label{eq:simplify6} \\
		\Hbf (x,r) & \leq C r^{m+3 - 2\delta_2} \label{eq:simplify7} \\
		\Gbf (x,r) &\leq C r^{-1} \Dbf (x,r) \label{eq:simplify8} \\
		|\partial_r \Dbf (x,r)| & \leq C r^{-1} \Dbf (x,r) \label{eq:simplify9} \\
		|\partial_r \Hbf (x,r)| &\leq C \Dbf (x,r)\, ,  \label{eq:simplify10}
	\end{align}
	Moreover, we have the estimates
	\begin{align}
		|\Dbf (x,r) - \Ebf (x,r)| & \leq C \boldsymbol{m}_0^{\gamma_4} r^{\gamma_4} \Dbf (x,r)\\
		|\partial_r \Dbf (x,r) - (m-2)r^{-1} \Dbf(x,r) - 2\Gbf (x,r)| &\leq C \boldsymbol{m}_0^{\gamma_4} r^{\gamma_4-1} \Dbf (x,r) \label{eq:simplify11}\\
		\partial_r \Ibf (x,r) &\geq - C \boldsymbol{m}_0^{\gamma_4} r^{\gamma_4-1}\, .
	\end{align}
\end{lemma}

The majority of the estimates in Lemma \ref{lem:simplify} follow directly from those in \eqref{e:bound-Lambda}, Lemma \ref{lem:firstvar} and Corollary \ref{cor:freqmono}, with the exception of the lower frequency bound in \eqref{eq:simplify1}. For the proof of this, we direct the reader to \cite{DLSk2}*{Lemma 4.1} (with the obvious difference that the limiting Dir-minimizer takes values in $\Ascr_Q$, not $\Acal_Q$ in the compactness procedure).

Before we go any further, we need to address the following. As well as knowing that $\Sbf \subset \mathbf{\Phi}(\mathbf{\Gamma})$, we will need to know that the frequency $\Ibf(x,0)$ is sufficiently large (namely, above the threshold $2-\delta_2$) for every $x\in\Sbf$. Indeed, this is the case; by studying the linearized problem that arises from a compactness procedure and ruling out the possibility that $\Ibf(x,0) = 1$ (since $x$  is a \emph{flat} singular point of $T$), one can show that $\Ibf(x,0)$ is always a positive integer strictly larger than 1 for each $x\in\Sbf$.

\begin{proposition}[\cite{DLHMSS_structure}*{Proposition 2.8}]\label{prop:highfreq}
	Let $T$, $\Mcal$, $N$ and $\gamma_4$ be as in Lemma \ref{lem:firstvar}. For every $x \in \Sbf$ and any sequence of scales $r_k \todown 0$, the following properties hold (up to subsequence):
	\begin{itemize}
		\item[(i)] $\bar{N}_{x,r_k} \coloneqq \frac{N(\mathbf{e}(x,r_k\cdot))}{\bar\Dbf(x,r_k)^{1/2}}$ converges strongly in $W^{1,2}_\loc$ to an $\Ibf(x,0)$-homogeneous Dir-minimizer $u:\pi_\infty \supset B_1 \to \Ascr_Q(\R)$ with $\eta\circ u = 0$ and $u(0) = Q\llbracket 0 \rrbracket$;
		\item [(ii)]  the frequency $\Ibf(x,0)$ is a positive integer strictly larger than 1.
	\end{itemize}
\end{proposition}
However, the above proposition does not guarantee that there are no points $y\in \mathbf{\Phi}(\mathbf{\Gamma})$ with $\Ibf(y,0) = 1$ nearby to points $x\in\Sbf$ (in fact, we expect any $x\in \Sbf$ to be an accumulation of such points $y$). With this information in mind, we let
\[
\Delta_Q N \coloneqq \set{x \in \Mcal}{N(x) = Q\llbracket 0 \rrbracket}, \qquad \Delta_Q^{2-\delta} N \coloneqq  \set{x \in \Mcal}{N(x) = Q\llbracket 0 \rrbracket, \ \Ibf(x,0) \geq 2-\delta}.
\]
We will use the same notation for Dir-minimizers $u: \R^m \supset \Omega \to \Ascr_Q(\R^n)$. Thus, Proposition \ref{prop:highfreq} enables us to say that for $T$, $\Mcal$, $N$ as in Assumption \ref{asm:modp2}, we have $\Sbf \subset \Delta_Q^{2-\delta} N$.

\section{Spatial frequency variations}\label{ss:variations}
Here, we control how much $N$ deviates from being homogeneous on average locally around a point $x$ between two scales, in terms of the radial frequency variation at $x$ between those scales. It is convenient to introduce the following terminology.

\begin{definition}\label{def:freqpinch}
	Suppose that $T$, $\Mcal$ and $N$ are as in Assumption~\ref{asm:modp2}. For $x \in \Bbf_{1}\cap\Mcal$ and any $\eta\, d (x, \mathbf{S}) < \rho \leq r \leq 1$, define the \emph{frequency pinching} $W_\rho^r(x)$ between $\rho$ and $r$ by
	\[
	W_\rho^r(x) \coloneqq |\Ibf(x,r) - \Ibf(x,\rho)|.
	\]
\end{definition}

\begin{proposition}\label{prop:distfromhomog}
	Suppose that $T$, $\Mcal$, $N$ are as in Assumption~\ref{asm:modp2}, let $\gamma_4$ be as in Lemma~\ref{lem:firstvar} and let $\Lambda>0$ be as in \eqref{e:bound-Lambda}. There exists $C = C(m,n,Q,\Lambda) > 0$ and $\beta = \beta(m,Q,\Lambda) >0$ such that the following estimate holds for every $x \in \Bbf_{1}\cap\Mcal$ and for every pair $\rho, r$ with $4\eta\, d (x, \mathbf{S}) < \rho\leq r < 1$. If we define 
	\begin{align*}
		\Acal_{\frac{\rho}{4}}^{2r}(x)  \coloneqq & \left(\Bbf_{2r}(x)\setminus \cl{\Bbf}_{\frac{\rho}{4}}(x)\right)\cap\Mcal
	\end{align*}
	then
	\begin{align*}
		&\int_{\Acal_{\frac{\rho}{4}}^{2r}(x)} \sum_i \left|DN_i(y)\frac{d(x,y)\nabla d(x,y)}{|\nabla d(x,y)|} - \Ibf(x,d (x,y)) N_i(y)|\nabla d(x,y)|\right|^2 \frac{\dd y}{d(x,y)} \\
		&\qquad\leq C \Hbf(x,2r) \left(W_{\frac{\rho}{4}}^{2r}(x) + \boldsymbol{m}_0^{\gamma_4} r^{\gamma_4}\log\left(\frac{4r}{\rho}\right)\right).
	\end{align*}
\end{proposition}

We refer the reader to \cite{DLSk2}*{Proposition 5.2} for the proof of Proposition \ref{prop:distfromhomog}. Since we just require the estimates in Lemma \ref{lem:simplify} to prove this proposition (in place of the estimates \cite{DLSk2}*{Lemma 4.1}), the proof remains completely unchanged in this setting. We will also require the following control on variations of the frequency in terms of frequency pinching.
\begin{lemma}\label{lem:spatialvarI}
	Suppose that $T$, $\Mcal$ and $N$ be as in Assumption~\ref{asm:modp2}, let $\gamma_4$ be as in Lemma \ref{lem:firstvar} and let $\Lambda > 0$ be as in \eqref{e:bound-Lambda}. Let $x_1,x_2 \in \Bbf_1 \cap \Mcal$ with $d(x_1,x_2) \leq \frac{r}{8}$, where $r$ is such that $8\eta\, \max\{d (x_1, \mathbf{S}), d (x_2, \mathbf{S})\} < r \leq 1$. Then there exists a constant $C= C(m,Q,\Lambda) > 0$ such that for any $z,y \in [x_1,x_2]$, we have
	\[
	|\Ibf(y,r) - \Ibf(z,r)| \leq C \left[\left(W_{\frac{r}{8}}^{4r}(x_1)\right)^{\frac{1}{2}} + \left(W_{\frac{r}{8}}^{4r}(x_2)\right)^{\frac{1}{2}} + \boldsymbol{m}_0^{\frac{\gamma_4}{2}} r^{\frac{\gamma_4}{2}}\right]\frac{d (z,y)}{r}\, .
	\]
\end{lemma}

The proof of Lemma \ref{lem:spatialvarI} relies on the following additional variation estimates and identities.

\begin{lemma}\label{lem:spatialvarDH}
	Let $T$, $\Mcal$ and $N$ be as in Assumption~\ref{asm:modp2} and let $x \in \Bbf_{1} \cap \Mcal$. Let $\eta\, d (x, \mathbf{S}) < \rho < r \leq 1$, and suppose that $v$ is a vector field on $\Mcal$. Then the following identities hold:
	\begin{align*}
		\partial_v \Dbf(x,r) &= \frac{2}{r} \int \phi'\left(\frac{d(x,y)}{r}\right) \sum_i \partial_{\nu_x} N_i(y)\cdot \partial_v N_i(y) \dd y  + O\left(\boldsymbol{m}_0^{\gamma_4}\right)r^{\gamma_4 -1}\Dbf(x,r)\\
		\partial_v \Hbf(x,r) &= - 2 \sum_i \int_\Mcal \frac{|\nabla d(x,y)|^2}{d(x,y)}\phi'\left(\frac{d (x,y)}{r}\right)\langle \partial_v N_i(y), N_i(y) \rangle \dd y\, . 
	\end{align*}
\end{lemma}

\begin{proof}[Proof of Lemma~\ref{lem:spatialvarDH}]
	The proof of these estimates is entirely analogous to those in \cite{DLSk2}*{Lemma 5.5}, so we omit many of the details. Indeed, notice that the identity \cite{DLSk2}*{(29)} still holds here, by decomposing the domain of the integral in $\Dbf(x,r)$ into the disjoint components of $\Bbf_r(x)\cap \Mcal_+$ and $\Bbf_r(x)\cap\Mcal_-$ as defined in \cite{DLHMS}, each of which is a relatively open set by~\cite{DLHMSlin}*{Corollary~2.8} (which also holds for a submanifold domain with $C^{3,\kappa}$-regularity). Note that the set $\Bbf_r(x)\cap\Mcal_0$ where $N=Q\llbracket 0 \rrbracket$ as defined in \cite{DLHMS} (not to be confused with our former terminology of this form) has $\Hcal^m$-measure zero.
	
	Now we test~\cite{DLS16blowup}*{(3.25)} with the vector field $X_i(q) = Y(\mathbf{p}(q))$ for 
	\[
	Y(y) \coloneqq \phi\left(\frac{d(x,y)}{r}\right) v,
	\]
	which satisfies the differential identities \cite{DLSk2}*{(32), (33)}, and exploit the excess decay of Proposition \ref{prop:excessdecaymodp} to establish the same estimates on the inner variational errors $\widetilde{\Err}_j^i$ therein (see \cite{DLHMS}*{(26.9),~(26.16),~(26.17),~(26.18)}) for this vector field here.
	
	The identity for $\partial_v\Hbf(x,r)$ is once again a simple computation, identical to that in the proof of~\cite{DLMSV}*{Proposition~3.1}, again combined with the domain decomposition into $\Mcal_+$ and $\Mcal_-$.
\end{proof}

Having established the validity of Lemma \ref{lem:spatialvarDH}, the proof of Lemma~\ref{lem:spatialvarI} follows in exactly the same way as that of \cite{DLSk2}*{Lemma 5.4}, yet again decomposing the domains of integration in $\Dbf(x,r)$, $\Hbf(x,r)$ into $\Mcal_+$ and $\Mcal_-$ when taking spatial derivatives of $\Ibf(x,r)$, and noting that the estimates remain unchanged.

\section{Quantitative spine splitting}\label{ss:splitting}
We will now demonstrate that under the assumption of approximate homogeneity around a collection of points spanning a given subspace, one achieves the existence of an approximate spine in that subspace, in a quantitative manner. We will be considering affine subspaces spanned by families of linearly independent vectors, so we introduce the following notation. Given an ordered set of points $X= \{x_0, x_1, \ldots, x_k\}$ we denote by $V (X)$ the affine subspace spanned by the vectors $\{x_1-x_0, x_2-x_0, \ldots, x_k -x_0\}$ and centered at $x_0$:
\begin{equation}\label{e:def-V(X)}
	V (X) = x_0 + \spn (\{(x_1-x_0), (x_2-x_0), \ldots, (x_k-x_0)\})\, .
\end{equation}

We recall the following definitions of quantitative linear independence and spanning from \cite{DLSk2}, where they are introduced in the integral currents framework.

\begin{definition}[\cite{DLSk2}*{Definition 6.1}]\label{def:LI}
	We say that a set $X = \{x_0, x_1,\dots,x_k\} \subset \Bbf_r(x)$ is $\rho r$-linearly independent if 
	\[
	d(x_i, V (\{x_0, \ldots, x_{i-1}\})) \geq \rho r \qquad \mbox{for all $i=1, \ldots , k$}
	\]
	We say that a set $F \subset \Bbf_r(x)$ $\rho r$-spans a $k$-dimensional affine subspace $V$ if there is a $\rho r$-linearly independent set of points $ X= \{x_i\}_{i=0}^k \subset F$ such that $V= V (X)$.
\end{definition}

\subsection{Compactness and homogeneity}
Before we proceed, we require the following invariance and unique continuation results, which are analogous to their counterparts in \cite{DLMSV} in the framework of $\Acal_Q$-valued Dir-minimizers, but we must verify that the presence of codimension one zeros in this setting does not pose an obstruction to their validity.
\begin{lemma}\label{lem:modphomog}
	Let $\Omega \subset \R^m$ be a connected open set and let $u: \Omega \to \Ascr_Q(\R^n)$ be a continuous map that is homogeneous about two points $x_1$ and $x_2$, with respective homogeneities $\alpha_1$ and $\alpha_2$.
	
	Then $\alpha_1 = \alpha_2$ and $u$ is invariant along the direction $x_2 - x_1$. Moreover, we have $x_1 +\spn\{x_2 - x_1\} \subset \Delta_Q u$.
\end{lemma}

\begin{lemma}\label{lem:modpuniquecont}
	Let $\delta \in (0,1)$, let $\Omega \subset \R^m$ be a connected open set and suppose that $u_1, u_2: \Omega \to \Ascr_Q(\R^n)$ are two homogeneous maps such that
	\begin{enumerate}[(a)]
		\item both $u_1$ and $u_2$ locally minimize the Dirichlet energy; \\
		\item there exists a non-empty open set $U \subset \Omega$ such that $u_1 \equiv u_2$ on $U$; \\
		\item\label{itm:nofreq1} for $j = 1,2$ we have $\Delta_Q u_j \equiv \Delta_Q^{2-\delta} u_j$.
	\end{enumerate}
	
	Then $u_1 \equiv u_2$ on $\Omega$.
\end{lemma}

\begin{proof}[Proof of Lemma \ref{lem:modphomog}]
	The proof of this follows by a very similar reasoning to the proof of \cite{DLMSV}*{Lemma 6.8}, which is the analogous result for classical $Q$-valued Dir-minimizers. However, one has to check that the argument is unchanged by the presence of the regions $\Omega_{\pm}$, separated by the points $x\in\Delta_Q u$ with $I_u(x) = 1$.
	
	The homogeneity of $u$ tells us that
	\[
	u^\pm(x) = \sum_i \left\llbracket |x-x_\ell|^{\alpha_\ell}  u_i\left(\frac{x-x_\ell}{|x-x_\ell|} + x_\ell\right) \right\rrbracket, \qquad \text{$x \in \Omega_{\pm}\sqcup \Omega_0$ respectively}, \quad \ell = 1,2.
	\]
	Recall that that each of $u^+$ and $u^-$ is a classical $\Acal_Q(\R^n)$-valued Dir-minimizer. We may thus extend each of these to $\R^m$ by homogeneity, and apply \cite{DLMSV}*{Lemma 6.8} to each one individually. The conclusion follows immediately.
\end{proof}

\begin{proof}[Proof of Lemma \ref{lem:modpuniquecont}]
	Observe that condition~\eqref{itm:nofreq1} tells us that for $j=1,2$, it holds that $\Omega \setminus \Delta_Q u_j$ is a affine subspace of dimension at most $m-2$, in light of the homogeneity assumption, combined with the knowledge that all one-dimensional $\Ascr_Q(\R^n)$-valued Dir-minimizers are locally superpositions of linear functions. Thus, $u_1$ and $u_2$ can be identified with classical (homogeneous) $Q$-valued Dir-minimizers that take values in $\Acal_Q$. This means that~\cite{DLMSV}*{Lemma~6.9} can be applied directly.
\end{proof}

The following lemma gives a quantitative notion of the existence of an approximate spine in $\Sbf$, provided that $N$ is (quantitatively) almost-homogeneous about an $(m-2)$-dimensional submanifold of the center manifold. It is entirely analogous to \cite{DLSk2}*{Lemma 6.2}, only posed in the mod$(p)$ setting.
\begin{lemma}\label{lem:noQpts}
	Suppose that $T$, $\Mcal$, $N$ are as in Assumption~\ref{asm:modp2}, let $x\in\Sbf$ and let $\rho,\tilde\rho,\bar\rho \in ]0,1]$ be given. There exists $\eps = \eps_{\ref{lem:noQpts}}(m,Q,\Lambda, \rho, \tilde\rho,\bar\rho) \in ]0,\eps_4^2]$ such that the following holds. Suppose that for some $r>0$,
	\[
	\Ebf^{no}(T, \Bbf_{2r}(x)) + (2r\Abf)^2 \leq \eps.
	\]
	Suppose that $X = \{x_i\}_{i=0}^{m-2} \subset \Bbf_r(x) \cap \Sbf$ is a $\rho r$-linearly independent set of points with
	\[
	W_{\tilde\rho r}^{2r}(x_i) < \eps \qquad \text{for each $i$}.
	\]
	Then $\Sbf \cap (\Bbf_r \setminus \Bbf_{\bar\rho r}(V (X))) = \emptyset$.
\end{lemma}
\begin{proof}
	We prove this by contradiction. We may without any loss of generality assume that $x=0$. Now, suppose that the statement of the lemma is false. Then we may find sequences $\eps_k \todown 0$, $r_k \todown 0$ and corresponding sequences of center manifolds $\Mcal_k$ and normalized normal approximations $\bar N_k$ with $\Hbf_{\bar N_k}(0,1) = 1$ for $T_{0,r_k}$. Letting $\Sbf_k \coloneqq \Sbf(T_{0,r_k})$, we in addition have a sequence of $(m-1)$-tuples of points $X_k \coloneqq \{x_{k,0}, x_{k,1},\dots, x_{k,m-2}\} \subset \Bbf_{1} \cap \Sbf_k$ such that
	\begin{enumerate}[(i)]
		\item $X_k$ is $\rho$-linearly independent for some $\rho \in ]0,1]$;
		\item\label{itm:pinch} $W_{\tilde\rho}^{2}(\bar N_k, x_{k,i}) \leq \eps_k \to 0$ as $k \to \infty$ for some $\tilde{\rho} \in ]0,1]$;
		\item there exists a point $y_k \in \Sbf_k \cap (\Bbf_1 \setminus \Bbf_{\bar\rho}(V(X_k)))$.
	\end{enumerate}
	
	We can thus proceed to use a compactness argument as in Proposition \ref{prop:highfreq} (see, also \cite{DLHMS}*{Section 28} or \cite{DLSk2}*{Section 2.2} in the integral currents framework) in order to deduce that
	\begin{enumerate}[(1)]
		\item $\Mcal_k \longrightarrow \pi_\infty$ in $C^{3,\kappa}$;
		\item there exists a Dir-minimizer $u:\pi_\infty \supset B_1 \to \Ascr_Q(\R)$ with $\boldsymbol{\eta} \circ u \equiv 0$ such that $\bar N_k \circ \mathbf{e}_k \longrightarrow u$ in $L^2$ and in $W^{1,2}_\loc$;
		\item The sequence $X_k$ converges pointwise to  $X_\infty = \{x_0,\dots,x_{m-2}\}$;
		\item The points $y_{k}$ converge pointwise to $y \in \bar B_1\setminus B_{\bar\rho}(V(X_\infty)) \subset \pi_\infty$ with $u(y) = Q\llbracket 0 \rrbracket$.
	\end{enumerate}
	By \cite{DLHMSS_structure}*{Theorem 3.6, Theorem 3.7} and a standard stratification argument, we know that $\dim_{\Hcal}(\Delta_Q^{2-\delta} u) \leq m-2$, since $H_u(0,1) = 1$ and $\eta \circ u = 0$, so $u$ cannot be a classical harmonic map with multiplicity $Q$. Moreover, $H_u(y, \tau) > 0$ for every $\tau \in (0,1)$, since otherwise we would contradict the dimension estimate on $\Delta_Q^{2-\delta} u$. This, in combination with~\eqref{itm:pinch} tells us that
	\[
	I_u(x_i, \tilde{\rho}) = I_u(x_i,2) \geq 2-\delta \qquad \text{for $i = 0,\dots,m-2$}.
	\]
	The monotonicity of the regularized frequency as defined in Section \ref{s:prelim} for $\Ascr_Q$-valued Dir-minimizers then tells us that $u$ is $\alpha_i$-homogeneous about $x_i$ within the annulus $B_2(x_i)\setminus B_{\tilde\rho}(x_i) \subset \pi_\infty$, for some $\alpha_i \geq 2$. Firstly, we may immediately deduce that $\alpha_i = \alpha$ for some fixed $\alpha \geq 2-\delta$ by iteratively applying Lemma \ref{lem:modphomog}, and also that $\Delta_Q u = \Delta_Q^{2-\delta} u$. We may then extend $u$ to an $\alpha$-homogeneous function $v$ about the $(m-2)$-dimensional affine subspace $V(X_\infty)$, in light of Lemma \ref{lem:modpuniquecont}.
	
	Since $y \notin V(X_\infty)$ and $u(y) = Q\llbracket 0 \rrbracket$, but $u$ is $\alpha$-homogeneous about $V(X_\infty)$, this implies that $u \equiv Q\llbracket 0 \rrbracket$ on $L \coloneqq x_0 + \spn\{x_{m-2} - x_0, \dots, x_1 - x_0, y-x_0\}$, and $I_u(\cdot, 0) \equiv \alpha\geq 2$ on the $(m-1)$-dimensional plane $L$. This however contradicts the dimension estimate on $\Delta_Q^{2-\delta} u$, thus allowing us to conclude.
\end{proof}

The following lemma, which is the mod$(p)$ analogue of \cite{DLSk2}*{Lemma 6.3}, tells us that it is enough to establish approximate homogeneity on a linearly independent set of points, in order to achieve approximate homogeneity in the affine subspace spanned by these points.

\begin{lemma}\label{lem:smallspatialvar}
	Suppose that $T$, $\Mcal$ and $N$ are as in Assumption~\ref{asm:modp2}, let $x\in\Sbf$ and let $\rho, \tilde\rho, \bar\rho \in ]0,1]$ be given. Then for any $\delta > 0$, there exists $\eps = \eps_{\ref{lem:smallspatialvar}} > 0$, dependent on $m,Q,\Lambda, \rho, \tilde\rho,\bar\rho,\delta$ for which the following property holds. Suppose that for some $r>0$ we have
	\[
	\Ebf^{no}(T, \Bbf_{2r}(x)) + (2r\Abf)^{2} \leq \eps.
	\]
	In addition, suppose that $X = \{x_i\}_{i=0}^{m-2} \subset \Bbf_r(x) \cap \Sbf$ is a $\rho r$-linearly independent set of points with
	\[
	W_{\tilde\rho r}^{2r} (x_i) < \eps \qquad \text{for every $i=0,\dots,m-2$}.
	\]
	Then for every $y_1, y_2 \in \Bbf_r(x) \cap \Bbf_{\varepsilon r} (V (X)) \cap \mathbf{S}$ and for every $r_1, r_2 \in [\bar\rho r, r]$ we have
	\[
	|\Ibf(y_1,r_1) - \Ibf(y_2,r_2)| \leq \delta\, .
	\]
\end{lemma}

\begin{proof} 
	We will once again proceed to argue by contradiction. We again assume that $x=0$ without loss of generality. If the statement of the lemma is false, we may find sequences $\eps_k \todown 0$, $r_k \todown 0$ and corresponding sequences of center manifolds $\Mcal_k$ and normalized normal approximations $\bar N_k$ with $\Hbf_{\bar N_k}(0,1) = 1$ for $T_{0,r_k}$, with a sequence of $(m-1)$-tuples of points $X_k \coloneqq \{x_{k,0}, x_{k,1},\dots, x_{k,m-2}\} \subset \Bbf_1\cap\Sbf_k = \Bbf_1\cap \Sbf(T_{0,r_k})$ such that
	\begin{enumerate}[(i)]
		\item The set $X_k$ is $\rho$-linearly independent for some $\rho > 0$;
		\item $W_{\tilde\rho}^{2}(\bar N_k, x_{k,i}) \leq \eps_k \to 0$ as $k \to \infty$ for some $\tilde{\rho} > 0$;
		\item\label{itm:freqgap} there exist points $y_{k,1}, \ y_{k,2} \in\Bbf_1\cap \Bbf_{\varepsilon_k} (V (X_k)) \cap \Sbf_k$ and corresponding scales $r_{k,i} \in [\bar\rho, 1]$ with
		\[
		|\Ibf_k(y_{k,1},r_{k,1}) - \Ibf_k(y_{k,2},r_{k,2})| \geq \delta > 0,
		\]
		where $\Ibf_k \coloneqq \Ibf_{\bar N_k}$.
	\end{enumerate}
	We may now use an analogous compactness argument to that in the proof of Lemma \ref{lem:noQpts} to conclude that, up to subsequence, we have
	\begin{enumerate}[(1)]
		\item $\Mcal_k \longrightarrow \pi_\infty$ in $C^{3,\kappa}$;
		\item $\bar N_k \circ \mathbf{e}_k \longrightarrow u$ in $L^2$ and in $W^{1,2}_\loc$, where $u$ is an $\Ascr_Q$-valued Dir-minimizer with $\boldsymbol{\eta} \circ u \equiv 0$;
		\item the collections of points $X_k$ converge pointwise to  $X_\infty = \{x_0,\dots,x_{m-2}\}$;
		\item the points $y_{k,i}$ converge pointwise to $y_i$ and the respective scales $r_{k,i}$ converge to $r_i \in [\bar\rho,1]$ for $i=1,2$.
	\end{enumerate}
	Proceeding as in the proof of Lemma~\ref{lem:noQpts}, we arrive at the conclusion $u \equiv Q\llbracket 0 \rrbracket$ on $x_0 + \spn\{x_{m-2}-x_0,\dots,x_1 - x_0\} = V(X_\infty)$ with the additional property that
	\[
	I_u(x_i, \tilde{\rho}) = I_u(x_i,2) \geq 2-\delta \qquad \text{for $i = 0,\dots,m-2$}.
	\]
	Thus, $I_u(y,\tau) \equiv \alpha \geq 2-\delta$ for any $y \in V(X_\infty)$ and any $\tau > 0$. On the other hand, since $r_{k,i}\in [\bar\rho,1]$ and $\bar\rho > \eta \min\{d(y_1,\Sbf), \eta d(y_2,\Sbf)\}$, we additionally have $\Ibf_k(y_{k,i}, r_{k,i}) \to I_u(y_i,r_i)$ for $i=1,2$, so the property~\eqref{itm:freqgap} is in contradiction with the homogeneity of $u$ about $V(X_\infty)$.
\end{proof}
\section{Jones' $\beta_2$ coefficient control}
This section is dedicated to controlling the ``mean flatness'' in a ball for a given Radon measure $\mu$ supported in $\mathbf{S}$, in terms of an $(m-2)$-dimensional $\mu$-weighted average of the frequency pinching, up to a lower order error term. We hence recall here the definition of Jones' $\beta_2$ coefficient (here we only consider the latter associated to $(m-2)$-dimensional planes), which is frequently used in many contexts when controlling the flatness (in an averaged $L^2$ sense) of a given set. It will enable us to measure the mean flatness of $\mu$ at a given scale around a given point.
\begin{definition}[\cite{DLSk2}*{Definition 7.1}]\label{def:beta2}
	Given a Radon measure $\mu$ in $\R^{m+n}$, we define the $(m-2)$-dimensional Jones' $\beta_2$ coefficient of $\mu$ as
	\[
	\beta_{2,\mu}^{m-2}(x,r) \coloneqq \inf_{\text{affine $(m-2)$-planes $L$}} \left[r^{-(m-2)} \int_{\Bbf_r(x)} \left(\frac{\dist(y,L)}{r}\right)^2 \dd\mu(y)\right]^{1/2}.
	\]
\end{definition}
The main result of this section is the following, which yields the desired control on the $\beta_2$ coefficient of a measure supported in $\Sbf$.

\begin{proposition}\label{prop:beta2control}
	There exist thresholds $\eta=\eta_{\ref{prop:beta2control}} (m) > 0$, $\eps=\eps_{\ref{prop:beta2control}} (\Lambda, m,Q, \eta)$, $\alpha_0 = \alpha_0(\Lambda,m,Q) > 0$ and $C(\Lambda,m,Q) > 0$ such that the following holds. Suppose that $T$, $\Mcal$ and $N$ satisfy Assumption~\ref{asm:modp2} with parameters $\varepsilon_4\leq\eps_{\ref{prop:beta2control}} $ and $\eta\leq \eta_{\ref{prop:beta2control}}$. Suppose that $\mu$ is a finite non-negative Radon measure with $\spt (\mu) \subset \Sbf$. Then for all $r \in ]0 ,1]$ and every $x_0 \in \Bbf_{r/8}\cap \mathbf{S}$ we have
	\[
	[\beta_{2,\mu}^{m-2}(x_0, r/8)]^2 \leq \frac{C}{r^{m-2}} \int_{\Bbf_{r/8}(x_0)} W^{4r}_{r/8}(x)\dd\mu(x) + C \boldsymbol{m}_0^{\alpha_0} r^{-(m-2-\alpha_0)}\mu(\Bbf_{r/8}(x_0)).
	\]
\end{proposition}

The proof of Proposition \ref{prop:beta2control} requires the following preliminary lemma regarding a characterization of $\Ascr_Q$-valued Dir-minimizers that are $(m-1)$-invariant.

\begin{lemma}\label{lem:modpinvariant}
	Let $A_{r,R}(\bar{z})\coloneqq B_R(\bar{z})\setminus \bar{B}_r(\bar{z})\subset \R^m$ and suppose that $u: B_R (\bar{z})\to \Ascr_Q(\R^n)$ is a non-trivial Dir-minimizer. Assume there is a ball $B \subset \Omega$ and a system of coordinates $x_1,\dots,x_m$ such that $u\big|_{A_{r,R}(\bar{z})}$ is a function of $x_1$ only. Then $u$ is a function of only $x_1$ on all of $B_R(\bar{z})$.
	
	Moreover, one of the following two alternatives holds:
	\begin{enumerate}[(i)]
		\item\label{itm:noQpts} $\Delta_Q u = \emptyset$;
		\item\label{itm:1homog} there is a one-homogeneous Dir-minimizer $v: \R \to \Ascr_Q(\R^n)$ such that $u(x) = v(x_1)$ for $x \in \Omega$.
	\end{enumerate}
\end{lemma}

\begin{remark}
	Note that in case~\eqref{itm:1homog} in Lemma~\ref{lem:modpinvariant}, $\Delta_Q u = \{x_1 = c\}$ for some $c \in \R$ and $I_u(x, \tau) = 1$ for every $x \in \Delta_Q u$ and every $\tau > 0$. 
\end{remark}

\begin{proof}[Proof of Lemma~\ref{lem:modpinvariant}]
	Fix a point $x \in \Reg u \cap A_{r,R}(\bar{z})$, with coordinates $x_1,\dots,x_m$. Then, by definition (see~\cite[Definition~10.1]{DLHMSlin}) there must be a neighbourhood $U \ni x$ and a sign $\epsilon \in \{+1,-1\}$ such that
	\[
	u(y) = \left(\sum_{i=1}^Q \llbracket u_i(y) \rrbracket, \epsilon\right) \qquad \text{for $y \in U$}.
	\]
	More precisely, this is because $\Delta_Q u$ is a relatively closed set of Hausdorff dimension at most $m-1$.
	
	This in particular implies that we may identify $u\big|_U$ with a classical $\Acal_Q$-valued Dir-minimizer. Moreover, the invariance of $u$ in $A_{r,R}(\bar{z})$ implies that $u_i(y) = (y_1 - c_i)v_i$ for some $c_i \in \R$, $v_i \in \R^n$, since one-dimensional classical $\Acal_Q$-valued Dir-minimizers necessarily have affine decompositions. In addition, for any $i \neq j$, either $u_i(x) \neq u_j(x)$ or $u_i \equiv u_j$ on $U$.
	
	We may thus rewrite $u$ in terms of distinct representatives as
	\[
	u(y) = \left(\sum_{i=1}^{Q'} k_i \llbracket u_i(y) \rrbracket, \epsilon\right) \qquad \text{for $y \in U$},
	\]
	where $k_i \in \N$, $\N \ni Q' \leq Q$ and $u_i(x) \neq u_j(x)$ for every $i \neq j$.
	
	Now define
	\[
	M(x) \coloneqq \sup\set{y_1 > x_1}{\text{$u_i(y) = u_j(y)$ for some $i \neq j$}},
	\]
	with the convention that $M(x) = +\infty$ if this set is empty.
	
	We will proceed to show that
	\begin{equation}\label{eq:uniquecont}
		u(y) = \left(\sum_{i=1}^{Q'} k_i \llbracket u_i(y) \rrbracket, \epsilon\right) \qquad \text{for $y \in \{x_1 < y_1 < M(x)\}\cap B_R(\bar{z})$}.
	\end{equation}
	Let $W$ be the maximal open set in $\{x_1 < y_1 < M(x)\}\cap B_R(\bar{z})$ in which~\eqref{eq:uniquecont} holds. Note that $W \supset U$. If $W \neq \{x_1 < y_1 < M(x)\}\cap B_R(\bar{z})$, we can choose a point $\xi \in \{x_1 < y_1 < M(x)\}\cap B_R(\bar{z})\cap\partial W$. Since $\xi_1 < M(x)$, we must have $u_i(\xi) \neq u_j(\xi)$ for every $i \neq j$. This means that we may apply the unique continuation for each single-valued Dir-minimizer $u_i$, to conclude that there is a neighbourhood $V \ni \xi$ on which~\eqref{eq:uniquecont} holds. This, however, contradicts the maximality of $W$, so indeed~\eqref{eq:uniquecont} holds on the entirety of $\{x_1 < y_1 < M(x)\}\cap B_R(\bar{z})$.
	
	Now, notice that if $M(x) < \sup\set{y_1}{y \in B_R(\bar{z})}$, then
	\[
	\{y_1 = M(x)\}\cap B_R(\bar{z})\subset \{u = Q\llbracket 0\rrbracket\}.
	\]
	This is due to the fact that $\Hcal^{m-1}\left(\{y_1 = M(x)\} \cap B_R(\bar{z})\right) > 0$, while $\dim_{\Hcal}\left((B_R(\bar{z}))_\pm \cap \Delta_Q u\right) \leq m-2$.
	
	Similarly, let
	\[
	L(x) \coloneqq \inf\set{y_1 < x_1}{\text{$u_i(y) = u_j(y)$ for some $i \neq j$}},
	\]
	with the convention that $L(x) = -\infty$ if this set is empty.
	
	Proceeding in exactly the same way as above, we conclude that
	\begin{equation}\label{e:rep}
		u(y) = \left(\sum_{i=1}^{Q'} k_i \llbracket u_i(y) \rrbracket, \epsilon\right) \qquad \text{for $y \in \{L(x) \leq y_1 \leq M(x)\}\cap B_R(\bar{z})$},
	\end{equation}
	and
	\[
	\{y_1 = L(x)\}\cap B_R(\bar{z})\subset \{u = Q\llbracket 0\rrbracket\}.
	\]
	Moreover, notice that if $M(x) < +\infty$, then $L(x) = - \infty$, and if $L(x) > -\infty$ then $M(x)= + \infty$. This is due to the fact that $u$ is non-trivial, and each $u_i$ is affine on $\{L(x) < y_1 < M(x)\}$. 
	
	Now there are two possibilities; either one of $L(x),M(x)$ lies in $[\bar{z}_1-R, \bar{z}_1 + R]$, or not. In the latter case, $B_r(\bar{z})\subset \{L(x) \leq y_1 \leq M(x)\}$ and so it is immediate that the representation formula \eqref{e:rep} holds in $B_R(\bar{z})$. In the former case, suppose without loss of generality that $L(x) \in [\bar{z}_1-R, \bar{z}_1 + R]$. However, since on $A_{r,R}(\bar{z})\cap \{y_1 < L(x)\}$, $u$ remains a function of $x_1$ only, we may again exploit the affine structure of $\Acal_Q$-valued Dir-minimizers to conclude that the representation formula \eqref{e:rep} holds in the entirety of $B_R(\bar{z})$. The dichotomy~\eqref{itm:noQpts} or~\eqref{itm:1homog} follows immediately in both cases.    
\end{proof}  

\begin{remark}
	In fact, the proof of Lemma \ref{lem:modpinvariant} demonstrates that the conclusion of the lemma holds true in any open, connected domain $\Omega \subset \R^m$ in place of $B_R(x_0)$, if $u$ is a function of $x_1$ only on an open subset $\Omega'\subset \Omega$ that contains a point $x$ with $u(x)=Q\llbracket 0 \rrbracket$. The author suspects that this more general version result may remain true even without the requirement that $\Omega'$ contains a point $x$ with $u(x)=Q\llbracket 0 \rrbracket$. However, since such a more general version of the result is not required here, we do not pursue this here.
\end{remark}

\begin{proof}[Proof of Proposition~\ref{prop:beta2control}]
	We may assume that $\mu(\Bbf_{r/8}) > 0$, since the desired estimate is otherwise trivial. The majority of this proof follows exactly as that of \cite{DLSk2}*{Proposition 7.2}. Indeed, letting $\Acal_{r/4}^{2r}(x_0) \coloneqq (\Bbf_{2r}(x_0)\setminus\Bbf_{r/4}(x_0))\cap\Mcal$ and proceeding in exactly the same manner as the proof therein, we arrive at the estimate
	\begin{align*}
		[\beta_{2,\mu}^{m-2}(x_0,r/8)]^2&\int_{\Acal_{r/4}^{2r}(x_0)}\sum_{j=1}^{m-1} |DN(z)\cdot \boldsymbol{\ell}_z (v_j)|^2 \dd z \\
		&\leq C r^{-(m-1)} \Hbf(x_0,2r)\left(\int_{\Bbf_{r/8}(x_0)} W^{4r}_{r/8}(x) \dd\mu(x) + \boldsymbol{m}_0^{\alpha_0} r^{\alpha_0} \mu(\Bbf_{r/8}(x_0))\right),
	\end{align*}
	for $\alpha_0 > 0$ sufficiently small, where $\boldsymbol{\ell}_z: T_{x_0} \mathcal{M} \to T_z \mathcal{M}$ is the linear map that corresponds to the differential $d \mathbf{e}_{x_0} |_{\zeta}$ of the exponential map $\mathbf{e}_{x_0}$ at the point $\zeta = \mathbf{e}_{x_0}^{-1} (z)$.
	
	It thus remains to check that
	\begin{equation}\label{eq:lowerbd}
		\int_{\Acal_{r/4}^{2r}(x_0)}\sum_{j=1}^{m-1} | DN(z)\cdot \boldsymbol{\ell}_z (v_j)|^2 \dd z \geq c(\Lambda)\frac{\Hbf(x_0,2r)}{r} ,
	\end{equation}
	for some $C(\Lambda) > 0$. We prove this via a contradiction and compactness argument, as usual. By scaling and translation invariance of the claimed bound, we may assume that $r=1$ and $x_0 = 0$. If~\eqref{eq:lowerbd} fails, then we can extract a sequence of currents $T_k$ with $\boldsymbol{m}_0^{(k)} \leq \eps_k^2 \to 0$, corresponding center manifolds $\Mcal_k$ in $\Bbf_1$, normalized normal approximations $\bar N_k$ with $\int_{\Bbf_2\setminus\bar{\Bbf}_{1}\cap\Mcal_k} |\bar N_k|^2 = 1$ and $\int_{\Bbf_1\cap\Mcal_k}|D\bar N_k|^2 \leq C\Lambda$, such that
	\begin{itemize}
		\item $\Mcal_k \to \pi_\infty$,
		\item $ \boldsymbol{\eta}\circ \bar N_k \to 0$,
		\item $\bar N_k(y_k) = Q\llbracket 0 \rrbracket$ for some $y_k \in \Bbf_{1/8}\cap \Mcal_k$ (since $\mu_{T_k}(\Bbf_{r/8}) > 0$),
	\end{itemize}
	but with
	\[
	\int_{\Bbf_{2}\setminus \Bbf_{1/4}\cap\Mcal_k}\sum_{j=1}^{m-1} |D\bar{N}_k(z)\cdot \boldsymbol{\ell}_z^k(v_j^k)|^2 \longrightarrow 0,
	\]
	for some choice of orthonormal vectors $\{v_1^k,\dots,v_{m-1}^k\}$. Up to subsequence, we can extract a limiting Dir-minimizer $u: \pi_\infty \supset B_{2} \to \Ascr_Q(\R)$ with
	\begin{itemize}
		\item $ \int_{B_2\setminus \bar{B}_{1}} |u|^2 = 1$,
		\item $\int_{B_1}|Du|^2 \leq C\Lambda$,
		\item $\boldsymbol{\eta}\circ u \equiv 0$,
		\item $u(y) = Q\llbracket 0 \rrbracket$ \quad for some $y \in B_{1/8}$,
	\end{itemize}
	but for which
	\[
	\int_{B_{2}\setminus \bar{B}_{1}}\sum_{j=1}^{m-1} |Du(z)\cdot v_j|^2 = 0
	\] 
	for orthonormal directions $v_j$ which are the (pointwise) limit of the directions $v^k_j$. Thus, arguing as in the proof of~\cite{DLMSV}*{Proposition~5.3}, we conclude that $u$ is a function of only one variable on $B_{2}\setminus \bar{B}_{1/4}$, and so Lemma~\ref{lem:modpinvariant} tells us that it is a function of only one variable the whole of $B_2$. Since $u(y) = Q\llbracket 0 \rrbracket$, we have $\dim_\Hcal(\Delta_Q u) \geq m-1$, which contradicts the fact that $u$ is non-trivial.
\end{proof}
\section{Coverings, Minkowski bound and rectifiability}

Now that we have the desired bounds on the $\beta_2$ coefficients as in Proposition \ref{prop:beta2control}, we are in a position to conclude the result of Theorem \ref{thm:modpmain}. The conclusion is achieved via an iterative covering procedure, originally appearing in \cite{NV_Annals}. It has since then further been used in \cite{DLMSV} in the context of classical multiple-valued Dirichlet-minimizing functions, followed by \cite{DLSk2} for a fixed normal approximation for an area-minimizing integral current of high codimension. 

The proofs of the results in this section are completely identical to those in \cite{DLSk2}*{Section 8}, relying only on the preceding results, which have now been established in this context, in the previous sections of this article. Thus, the proofs are omitted here, and we instead refer the reader to \cite{DLSk2}.

We begin with the following covering lemma, which is the analogue of \cite{DLSk2}*{Lemma 8.1}

\begin{lemma}\label{lem:cover1}
	Let $\rho \leq \frac{1}{100}$, let $\sigma < \tau < \frac{1}{8}$ and let $\eta = \eta_{\ref{prop:beta2control}} > 0$. There exists $\varepsilon_4 = \varepsilon_4(\Lambda,m,Q,\alpha_0) > 0$ sufficiently small such that the following holds. Suppose that $T$ is as in Assumption \ref{asm:modp2} for these choices of $\eta$ and $\eps_4$. Let $x\in \mathbf{S}\cap\Bbf_{1/8}$, let $D \subset \Sbf \cap \Bbf_\tau (x)$ and let $U \coloneqq \sup_{y \in D} \Ibf(y,\tau)$.
	
	Then there exists $\delta = \delta_{\ref{lem:cover1}}(m,Q,\Lambda,\rho) > 0$, a dimensional constant $C_R=C_R(m) > 0$ and a finite cover of $D$ by balls $\Bbf_{r_i}(x_i)$ such that
	\begin{enumerate}[(a)]
		\item\label{itm:covera} $r_i \geq 10\rho\sigma$;
		\item\label{itm:coverb} $\sum_i r_i^{m-2} \leq C_R \tau^{m-2}$;
		\item\label{itm:coverc} For every $i$, either $r_i \leq \sigma$ or
		\[
		F_i \coloneqq D \cap \Bbf_{r_i}(x_i) \cap \set{y}{\Ibf(y,\rho r_i) \in (U -\delta, U+\delta)} \subset \Bbf_{\rho r_i}(V_i),
		\]
		for some $(m-3)$-dimensional subspace $V_i \subset \mathbb R^{m+n}$.
	\end{enumerate}
\end{lemma}

\begin{remark}[Heirarchy of parameters]
	The parameters $\varepsilon_4$ and $\eta$ of Assumption \ref{asm:modp2} are initially taken to be small enough so that we can apply Proposition \ref{prop:beta2control}. Then, $\eps_4$ is further decreased if necessary, to ensure that $\boldsymbol{m}_0^{\alpha_0}$ falls below a desired small dimensional constant, in order to absorb a suitable error term (see the proof of \cite{DLSk2}*{Proposition 7.2}). Lemma \ref{lem:cover1} will then be used to prove the following additional efficient covering result, entirely analogous to \cite{DLSk2}*{Proposition 8.2}, where the parameter $\rho$ will be chosen smaller than a geometric constant depending only on $m$.
\end{remark}

\begin{proposition}\label{prop:cover2}
	Let $\eta =\eta_{\ref{prop:beta2control}} > 0$ and let $\eps_4 > 0$ be as in Lemma \ref{lem:cover1}. There exist $\delta = \delta_{\ref{prop:cover2}}(m,Q,\Lambda)$, a scale $\tau = \tau(m,Q,\Lambda,\delta) < \frac{1}{8}$ and a dimensional constant $C_V = C_V(m) \geq 1$ such that the following holds.
	
	Assume that $T$ is as in Assumption \ref{asm:modp2} for these choices of $\eta$ and $\eps_4$. Suppose that $x \in \Sbf\cap\Bbf_{1/8}$ and let $D \subset \Sbf\cap\Bbf_\tau(x)$ and $U \coloneqq \sup_{y \in D} \Ibf(y,\tau)$. Then, for every $s\in ]0,\tau[$, there exists a finite cover of $D$ by balls $\Bbf_{r_i}(x_i)$ with $r_i \geq s$ and a decomposition of $D$ into sets $A_i \subset D$ such that
	\begin{enumerate}[(a)]
		\item $A_i \subset D \cap \Bbf_{r_i}(x_i)$;
		\item\label{itm:packing} $\sum\limits_i r_i^{m-2} \leq C_V \tau^{m-2}$;
		\item For every $i$ we have either $r_i = s$ or
		\[
		\sup_{y \in A_i}\Ibf(y,r_i) \leq U - \delta.
		\]
	\end{enumerate}
\end{proposition}

\subsection{Conclusion of Theorem \ref{t:modpreduction}(v)}\label{s:concl}
The conclusion of Theorem \ref{t:modpreduction}(v) now follows from Proposition~\ref{prop:cover2} by exactly the same reasoning as that in \cite{DLSk2}*{Section 8.3}. We therefore do not include the details here, and refer the reader to the argument therein.

\end{document}